\documentclass[a4paper,10pt]{article}

\usepackage[Sonny]{fncychap}
\usepackage[export]{adjustbox}
\usepackage[normalem]{ulem}
\usepackage{framed}
\usepackage{lscape}
\usepackage{multirow}
\usepackage{fancyhdr} 
\usepackage{vmargin}
\usepackage{afterpage}
\usepackage[dvipsnames]{xcolor}

\usepackage[center]{titlesec}
\usepackage{imakeidx}
\usepackage[pdftex]{hyperref} 
\titleformat{\section}[hang]{\normalfont\scshape\centering}{\thesection.}{1em}{}
\titleformat{\subsection}[hang]{\normalfont}{\thesubsection.}{1em}{}
\usepackage{comment}
\usepackage{hyperref}
\hypersetup{
    colorlinks=true,
    linkcolor=blue,
    filecolor=magenta,      
    urlcolor=cyan,
    pdftitle={Overleaf Example},
    pdfpagemode=FullScreen,
    }

\usepackage{lmodern}
\usepackage[T1]{fontenc}
\usepackage[utf8]{inputenc}

\usepackage{amsmath}
\usepackage{amssymb}
\usepackage{amsthm}
\usepackage{mathrsfs}

\theoremstyle{definition}
\newtheorem{thm}{Theorem}[section]
\newtheorem{thm2}{Theorem}
\newtheorem{defi}{Definition}[section]
\newtheorem{prop}{Proposition}[section]
\newtheorem{remark}{Remark}[section]
\newtheorem{lem}{Lemma}[section]

\newcommand{\R}{\mathbb{R}} 
\newcommand{\Z}{\mathbb{Z}} 
\newcommand{\N}{\mathbb{N}} 

\newcommand{\conj}[2]{\left\{#1\colon#2\right\}} 
\newcommand{\fun}[3]{#1\colon#2\rightarrow#3} 
\newcommand{\abs}[1]{\left|#1\right|} 
\newcommand{\norm}[2]{\left\|#1\right\|_{#2}} 
\renewcommand{\i}{\texttt{i}} 
\renewcommand{\d}{\textup{d}} 
\renewcommand{\div}{\textup{div}\,} 
\newcommand{\curl}[1]{\textup{curl}\,(#1)} 
\renewcommand{\epsilon}{\varepsilon} 
\newcommand{\pib}[1]{{#1}^{\leq1}} 
\newcommand{\pob}[1]{{#1}^{>1}} 

\title{\textsf{On the topology of the magnetic lines of large solutions to the Magnetohydrodynamic equations in $\R^3$}}

\author{\textsc{R. Lucà}\footnote{ \small \texttt{renato.luca@univ-orleans.fr} -  \textsc{Institut Denis Poisson, UMR 7013, CNRS}, 
		Université d’Orléans, Bâtiment de Mathématiques, rue de
		Chartres, F-45100 Orléans.} and \textsc{C. Peña}\footnote{ \small \texttt{cpena@bcamath.org} - \textsc{Basque Center for Applied Mathematics (BCAM)}, Alameda de Mazarredo 14, 48009, Bilbao -- \textsc{Universidad del País Vasco / Euskal Herriko Unibertsitatea (UPV/EHU)}, Barrio Sarriena s/n, 48940 Leioa.}}
\date{}

\begin{document}

\maketitle

\numberwithin{equation}{section}

\noindent
\textsc{Abstract.} The purpose of this article is twofold: first, we introduce a new class of global strong solutions to the magnetohydrodynamic system in $\R^3$ with initial data $(u_0,b_0)$ of arbitrarily large size in any critical space (to say, $\norm{u_0}{\dot{B}_{\infty,\infty}^{-1}} \simeq M \simeq \norm{b_0}{\dot{B}_{\infty,\infty}^{-1}}$ with $M\gg1$). To do so, we impose a smallness condition on the difference $u_0-b_0$. 
Secondly, we use this result to prove the existence of global solutions to the system, starting from a suitable class of large (in critical spaces) initial data, which present magnetic reconnection. With this, we mean a change of topology of the integral lines of the magnetic field under the time evolution. The proof of this fact relies on comparing the number of hyperbolic critical points of the magnetic field at different times.


\section{Introduction}
The Cauchy initial value problem for the Magnetohydrodynamic equations is
\begin{equation}\label{eq:MHD}\tag{MHD}
    \left\lbrace\begin{aligned}
        &\partial_tu +(u\cdot\nabla)u +\nabla p =\nu\Delta u + (b\cdot\nabla)b\,, \\
        &\partial_tb + (u\cdot\nabla)b  = (b\cdot\nabla)u + \eta\Delta b\,, \\
        &\div u = \div b = 0\,, \\
        &u(0,\cdot) = u_0\,, \; b(0,\cdot) = b_0\,,
    \end{aligned}\right.
\end{equation}
where $(t,x)\in[0,T]\times\R^3$ for some $T\in(0,\infty]$, $\nu>0$ is the (constant) viscosity of the fluid, $\eta>0$ is the (constant) resistivity, and we have as unknowns: the velocity field $u\colon[0,T]\times\R^3\rightarrow\mathbb{R}^3$, the magnetic field $b\colon[0,T]\times\R^3\rightarrow\mathbb{R}^3$ and the pressure $p\colon[0,T]\times\R^3\rightarrow\mathbb{R}$. The system \eqref{eq:MHD} describes the time evolution of the velocity and the magnetic fields of charged viscous incompressible fluids without external forces, such as highly conducting plasmas. 

\subsection{Global existence of solutions}

The \eqref{eq:MHD} model is concerned with the evolution of a fluid in the presence of a magnetic field, and it is used in the study of plasma physics, geophysics or astrophysics among other fields. A well-explained physical background for this system can be found in \cite{Da01}. From the mathematical point of view, local existence (and global existence for small data) of strong solutions in the Sobolev scale is known since the pioneering work of Duraut and Lions \cite{DL72, ST83}. 
For more recent developments we refer to \cite{CDL14, FMRR14, JZ16} and the references therein. Whether the local solutions may develop a singularity or not in finite time is still an outstanding open problem (note that when $b = 0$ the MHD system reduces to the Navier-Stokes one, so 
the difficulties of the latter also appear in the former).

\medskip

In this paper we are interested in obtaining global strong solutions for \textit{large initial data}. In this setting, critical spaces play an important role:
\begin{defi}\label{def:crit.spaces}
	Let $(u,p)$ be a solution to \eqref{eq:MHD} with initial data $(u_0,b_0)$ and let $\lambda>0$. Then, the functions
	\begin{equation}\label{eq:rescaling_MHD}
		u^{\lambda}(t,x) = \lambda u(\lambda^2t, \lambda x), \quad b^{\lambda}(t,x) = \lambda b(\lambda^2t, \lambda x), \quad p^{\lambda}(t,x) = \lambda^2 p(\lambda^2t, \lambda x)
	\end{equation}
	solve \eqref{eq:MHD} with initial data
	\begin{equation*}
		u_0^{\lambda}(x) = \lambda u_0(\lambda x), \quad b_0^{\lambda}(x) = \lambda b_0(\lambda x), \quad p_0^{\lambda}(x) = \lambda^2 p(\lambda x).
	\end{equation*}
	A Banach space is critial for the initial data $(u_0,b_0)$ if the relative norm is invariant under the scaling \eqref{eq:rescaling_MHD}.
\end{defi}
\begin{remark}
	When the domain is $\R^3$, the solution to \eqref{eq:MHD} share the same scaling properties as the ones to Navier-Stokes (putting $b=0$).
\end{remark}
Some examples of critical spaces for \eqref{eq:MHD} in $\R^3$ are
\begin{equation*}
	\dot{H}^{1/2} \hookrightarrow L^3 \hookrightarrow \dot{B}_{p,\infty}^{-1+3/p} \hookrightarrow BMO^{-1} \hookrightarrow \dot{B}_{\infty,\infty}^{-1},
\end{equation*}
with $\dot{B}_{\infty,\infty}^{-1}$ being the largest translation invariant critical space for \eqref{eq:MHD} (see Definition \ref{def:Besov.spaces}). Global in time existence for \textit{small initial data} has been stablished in these spaces (by fixed-point arguments in the Duhamel formulation of the problem) up to $BMO^{-1}$ in \cite{MYZ07}, as well as ill-posedness in $\dot{B}_{\infty,\infty}^{-1}$ (see \cite{CD20, DQS11}). On the other hand, the large data problem is still widely open and has atractted the attention of many researchers.

\medskip

In this direction, some efforts have been focused on exploiting the interactions between the velocity and the magnetic fields. 
This idea has its origins in the pioneering work of Alfvén \cite{Al42} and has been further developed by other authors \cite{BSS88, HXY18} (and references therein).
In particular, in \cite{HHW14} the authors used the Els\"asser change of variables in order to prove global well-posedness for a suitable class of initial data. Due to the symmetric structure of the system \eqref{eq:MHD}, this change of variables has played an important role in its study, since it relates the original equations with a coupling of transport equations (see \cite{Sm88, CF21, CF23} and references therein). Our first result follows the same philosophy, although we rather study the differences of the field solutions and the corresponding linear propagators, taking advantage of suitable cancellations between the magnetic and the velocity fields in order to introduce a new class of (strong) global large solutions. Moreover, we do this at the level of Sobolev scales including the critical space $L^2(\R_+;L^{\infty}(\R^3))$ for the system. More precisely, we prove the following global existence and uniqueness result:
\begin{thm2}[\textbf{Global strong solutions}]\label{fdhsidhfb}
	Let $u_0,b_0\in H^r(\R^3)$, for some $r\geq3$, be such that
	\begin{equation*}
		\norm{u_0-b_0}{H^r} \lesssim\rho \quad\text{and}\quad \norm{e^{\eta t\Delta}u_0}{L_t^2W_x^{k,\infty}} + \norm{e^{\eta t\Delta}b_0}{L_t^2W_x^{k,\infty}}  \lesssim 1+N^k, \quad \forall 0\leq k\leq r,
	\end{equation*}
	for some constant $N\geq1$ and some (small) constant $0<\rho\leq cN^{-r-1}$ with $c>0$ sufficiently small. Then, there exists a unique global regular solution $(u,b)$ to \eqref{eq:MHD} with $\nu = \eta$ that satisfies
	$$\norm{u(t)-e^{\eta t\Delta}u_0}{H^k} + \norm{b(t)-e^{\eta t\Delta}b_0}{H^k} \leq C\rho N^k, \quad\forall 0\leq k\leq r, \quad\forall t\geq0.$$
\end{thm2}
\begin{remark}\label{rem:resistivity_equal_viscosity}
	We have focused on the case in which the viscosity and the resistivity are equal. In our proofs it is crucial to use some cancellations of large terms that appear thanks to the symmetries of the system (see for instance \eqref{eq:rewrite.lin.terms.level0v}, \eqref{eq:rewrite.lin.terms.level0h}, \eqref{eq:lin.term.levelk}). In order to maintain them, we must set the viscosity equal to the resistivity, $\nu=\eta$; a physical justifications of this choice can be found in \cite[Remark 2.4]{HHW14}. For a deeper understanding of the physical setting we suggest looking into \cite{JV18, La20} and references therein. A perturbative analysis may be implemented in order to extend the result to the case in which $|\nu - \eta|$ is sufficiently small, but it is left as an open problem to extend this result to the case where $\nu\neq\eta$.
\end{remark}

As mentioned before, the proof of this result is based on a perturbative analysis of the solution with respect to its 
linear propagator, that is, the evolution of the initial data under the heat semigroup. We obtain global existence of 
regular solutions via energy estimates for the perturbation $(u- e^{t\Delta}u_0,b-e^{t\Delta}b_0)$. Moreover, as in \cite{CCL25}, we are able to quantify the error in such a way that it depends only polynomially on the $L^2((0,T);W^{r,\infty})$ norm of the linear evolution $(e^{t\Delta}u_0,e^{t\Delta}b_0)$, while no higher order derivatives appear in the exponential dependence, which is only on the $L^2((0,T);L^\infty)$ norm. This will be crucial in order to implement this result in the context of the second problem that we study in this paper, that is to give examples of magnetic reconnection for large solutions (or, more precisely, for solutions evolving from large data).

\medskip

A choice of initial data which satisfies the assumption of our theorem is the following:
$$
u_0:=M \text{curl}(\phi B_N),  \quad b_0:=M\curl{\phi B_N}+\rho \curl{w}
$$
where $M>0$, $\phi, w \in H^{r+1}(\R^3)$ (with this we mean that each component of $w$ belongs to 
$H^{r+1}(\R^3)$) and $B_N$ is a Beltrami field of frequency $N > 1$, namely a solution of 
$$ \curl{B_N} = N B_N.$$
One can show (see \cite{CL24}) that $\| u_0 \|_{\dot{B}^{-1}_{\infty, \infty}} \simeq \| b_0 \|_{\dot{B}^{-1}_{\infty, \infty}} \simeq M$,
thus the initial data are large in any critical space (if we choose $M$ large). Moreover, it is easy to check that this class of 
initial data does not fit the smallness assumption of the well-posedness result in \cite{HHW14}. Again, our interest in this class of initial data 
is motivated by the fact that we use it to give examples of magnetic reconnection (for large solutions).

\subsection{Magnetic reconnection}
Our second goal is to study topology of the \textit{magnetic lines}. We define a 
\textit{magnetic line} as an integral line of $b$, namely a curve $\gamma$ on $\R^3$ which solves the ODE
$$\Dot{\gamma}(s)=b(t,\gamma(s)).$$ 
Now we can give the definition of the problem:
\begin{defi}
    A solution $(u,b)$ to \eqref{eq:MHD} shows \textit{magnetic reconnection} if there exist two times $t_1\neq t_2$ such that there is no homeomorphism of $\R^3$ mapping the set of integral lines of $b(t_1,\cdot)$ into the set of integral lines of $b(t_2,\cdot)$.
\end{defi}
\begin{remark}
    The analogous of this problem for the Navier-Stokes equations is called \textit{vortex reconnection}, where instead of the integral lines of the magnetic field, one studies those of the vorticity.
\end{remark}

An observation which goes back to Alfvén \cite{Al42} (\textit{frozen-in-flux theorem}) is that in the non-resistive case ($\eta=0$), the magnetic lines evolve with the flow (associated to the vector field $u$) for sufficiently regular solutions. In particular, there is no change in their topology, so magnetic reconnection is forbidden. On the other hand, if $\eta >0$, one expects that reconnection of magnetic lines must happen.
This concept is important from a physical point of view, since it is related to phenomena such as solar flares, coronal mass ejections and other astrophysical events which involve the transformation of the magnetic energy into other types of energy.

\medskip

Although there are experimental evidences of this (\cite{GPWXH22}, \cite{T09}), a mathematical framework for the phenomenon of reconnection had not been constructed yet. On the other hand, we are at least able to provide examples of solutions for which the topology of the magnetic lines (or of the vortex lines, in the case of the Navier--Stokes equation) changes over the evolution. The first analytic example of reconnection was given in \cite{ELP17} for Navier-Stokes. The authors considered a class of periodic initial data giving rise to a solution that changes the topology of its vortex lines on $\mathbb{T}^3$. Following the same spirit, \cite{CCL25,Ci24} proved magnetic reconnection on the torus for \eqref{eq:MHD}. To prove the (vortex/magnetic) reconnection in the torus, the contractibility of the integral lines of the field is used as 
a topological constraint that is broken in the non ideal case. The construction of the initial data in these results also relies on some deep topological properties of steady solutions to the Euler equations proved in \cite{EP12,EP15}. Finally, in \cite{CL24} the authors proved the phenomenon of vortex reconnection for the Navier-Stokes equations on the whole space $\R^3$ for small initial data. In this case, they used as a topological constraint the number of critical points of the vorticity (note that regular curves on $\R^3$ are always contractible). Recall that a hyperbolic critical point of a vector field $v$ is a point $p\in\R^3$ such that $v(p)=0$ and the Jacobian matrix of $v$ at $p$ has no eigenvalues with zero real part.

\medskip

Considering the state of the art of the problem, the next step is to build analytic examples giving rise to magnetic reconnection for \eqref{eq:MHD} defined on the full space $\R^3$. Interestingly we will be able to exhibit a class of {\it large} solutions for which the topology of the 
magnetic lines changes under the evolution (even for time which are arbitrarily small). In relation to this, it is worth to point out that there are still no results about vortex reconnection for large solutions of the Navier--Stokes equations.   
Thus here  we build an example of initial data (inspired on the one in \cite{CL24}) but of arbitrary size in any critical space, and we prove the existence of a (unique) global strong solution starting from this example, that actually shows magnetic reconnection: 
\begin{thm2}[\textbf{Magnetic reconnection}]
	Given any constants $M,\,T>0$, there exist two smooth (finite energy) divergence free vector fields $u_0,\,b_0$ in $\R^3$ such that $\norm{u_0}{\Dot{B}_{\infty,\infty}^{-1}} \simeq \norm{b_0}{\Dot{B}_{\infty,\infty}^{-1}} \simeq M$ and \eqref{eq:MHD} admits a (unique) global strong solution $(u,b)$ with initial datum $(u_0,b_0)$, such that the magnetic lines at time $t=0$ and $t=T$ are not topologically equivalent, meaning that there is no homeomorphism of $\R^3$ into itself mapping the magnetic lines of $b(0,\cdot)$ into those of $b(T,\cdot)$.
\end{thm2}

\begin{remark}
With finite energy we mean that $u_0, b_0 \in L^{2}(\R^3)$. Thus, by energy dissipation, the same is true for 
$u(t), b(t)$ at any $t >0$. 
\end{remark}

We give a sketch of the main ideas to prove reconnection. The aim is to show a change in the topology of the magnetic lines over the evolution. To do so, we use as a topological constraint the number of hyperbolic critical points of the field, which is invariant under homeomorphisms: we count them at different times to show that it actually changes. Thus, magnetic reconnection must have happened in the meanwhile. This "counting zeros" strategy has already been used in other works concerning the topology of the solutions to \eqref{eq:MHD} (see \cite{CCL25,CL24,Ci24}) 
as well as in the context of the Navier--Stokes equations (see \cite{EP25,LS12a,LS12b}).

\medskip

To build an analytic example of initial data in $\R^3$ giving rise to a solution presenting magnetic reconnection, we cannot follow the strategy of \cite{CCL25} in $\mathbb{T}^3$, where they consider initial data of the form
$$MB_0+\delta B_1$$
for some Beltrami fields $B_0, B_1$. Indeed, in $\R^3$ this solutions would have infinite energy. To overcome this difficulty, we follow the spirit of \cite{CL24}: we localize the Beltrami field with a bump function and put the curl operator in front to preserve the incompressibility condition. In the case of the Navier-Stokes equations, in order to control the evolution of such an initial datum it is necessary to impose a suitable smallness condition on its size (for instance in a sufficiently regular Sobolev scale). The novelty of this paper is that, in the context of the MHD equations, we will rather impose a smallness condition on the difference $u_0 - b_0$ between velocity and magnetic field, but both $u_0$ and $b_0$ may be 
arbitrarily large. This is the key idea that allows us to prove magnetic reconnection for large solutions.  
Moreover, in order to control the evolution we must use crucially the fact that the initial velocity and magnetic fields are very close. It is worth mentioning that this is related to (or at least reminiscent of) the Els\"asser change of variables.

\medskip

To be more precise about the objects of our proof and the strategy that we follow, we consider the (\textit{large} when $M$ is large) initial data
$$
u_0:=M \text{curl}(\phi B_N),  \quad b_0:=u_0+\rho e^{-\eta T\Delta}\curl{\psi W}
$$
and we choose the parameters in such a way that:
\begin{itemize}
	\item $M>0$ will be the size of the initial data, $T>0$ will be the time of reconnection, $\rho>0$ is a small parameter proportional to the size of $u_0-b_0$ in suitable norms and $N>0$ is a large parameter that will be used to 
	bump the magnetic fiels during the evolution,
	\item $B_N$ will be a high frequency Beltrami field in $L^\infty$ with eigenvalue $N$ (in fact an ABC flow),
	\item $W$  will be another $L^\infty$-field with null points,
	\item $\phi$ and $\psi$ will be functions that decrease polynomially and exponentially at infinity, respectively,
	\item $b_0$ will have no critical points, while $\curl{\psi W}$ will have $(0,0,0)$ as a hyperbolic critical point, and these properties are robustly stable under $C^1$-perturbations and time perturbations,
	\item we will prove that the solution $b(T)$ will be close to $\curl{\psi W}$ in the $C^1$-norm, so the magnetic reconnection must have happened between $t=0$ and $t=T$. To do so we will crucially use Theorem \ref{fdhsidhfb}. 
\end{itemize}
\begin{remark}
	This reconnection scenario is topologically robust, meaning that small $C^1$ perturbations of the initial data will still give a solution presenting this phenomenon. This is a consequence of the implicit function theorem, and it will be clear by the end of the proof. Similarly, the absence of critical points persists for sufficiently small times and the presence of a hyperbolic critical point persist for times sufficiently close to $T$ (again this second fact follows by the implicit function theorem).	
\end{remark}

We want to point out that we are not able to use this strategy in order to prove vortex reconnection for large solutions of the Navier-Stokes equation. One could be indeed tempted to use localization of Beltrami fields in order to do so, however this lead to some issue that we are still not able to solve. In order to explain that, we recall the well-posedness result for localization of Beltrami fields from \cite{CL24} (that relies upon the nonlinear smallness condition introduced in \cite{CG09}).
 
\begin{thm}[\cite{CL24}, Theorem 2] \label{fdjskduhgfdsjhdf}
    Let $M>0$ and $q\in(3,\infty)$. Let $B_\lambda\in L^q\cap L^\infty(\R^3)$ be a Beltrami field of frequency $\lambda\neq0$ such that $\norm{B_\lambda}{L^\infty}=1$. Let $\phi\in W^{s,1}\cap W^{s,\infty}(\R^3)$ be a positive function. We consider the divergence free vector field
    $$u_0:=M\curl{\phi_LB_\lambda}, \quad \phi_L(\cdot):=\phi(\cdot/L), \quad M,L>0.$$
    Then, for all sufficiently large values of the regularity $s$ and of the dilation parameter $L$, the following holds:
    \begin{itemize}
        \item every scaling invariant norm of $u_0$ is large if $M$ is large. In fact, $\norm{u_0}{\Dot{B}_{\infty,\infty}^{-1}}\simeq M$;
        \item there exists a unique strong solution $u$ to the Navier-Stokes equations with initial datum $u_0$
        (the solution is smooth and has finite energy).
    \end{itemize}
\end{thm}

The difference between our result and the result in \cite{CL24} (Theorem 1) can be understood looking at the smallness condition that is imposed. They use the fact that, when the parameter $L$ is large, the function $\phi_L$ is close to one, that is, $u_0$ is almost the Beltrami field, and they exploit the fact that $B_\lambda$ is a stationary solution to Euler equations; in particular,
$$\mathbb{P}(B_\lambda\cdot\nabla)B_\lambda=0, \quad\mathbb{P}\text{ being the Leray projection.}$$
But in order to take advantage of this fact, it is crucial for them to impose some extra integrability over the field. In our case, we avoid this condition using the interactions between the velocity and the magnetic fields. 
More precisely, we use the fact that, even if both the magnetic field and the velocity field are large, their difference is small at the initial time. This induces a crucial cancellation when we study the system associated to the difference between the (local in time) solution and the associated linear evolution. Taking advantage of this cancellation one can propagate the smallness assumption for larger time, achieving in this way global well-posedness. In conclusion, both methods are quite different and the smallness conditions are not comparable.

\medskip

We conclude this introduction with some comments about vortex reconnection. Let first notice that, despite Theorem \ref{fdjskduhgfdsjhdf} covers a large class of Beltrami fields, the necessity of $L^q$ integrability for some $q \neq \infty $ leaves out an important type of Beltrami fields: the ABC flows. Note that in this paper (as well as in \cite{CL24}) a specific example of ABC flow with no critical points is crucially used for showing reconnection. In order to use this method to prove vortex reconnection for large solutions of the Navier--Stokes equations on $\R^3$ one should be  able to solve one of the two (nontrivial) following problems:
\begin{itemize} 
\item  Extend Theorem \ref{fdjskduhgfdsjhdf} to $q = \infty$ (in such a way that also localizations of ABC flows are 
taken into account). 
\item Give an example of an $L^p(\R^3)$, $p \in (3 , \infty)$ Beltrami field that has no critical points. We remark that this is challenging since Beltrami fields on $\R^3$ typically have several critical points, as proved in \cite{EPR23}.
\end{itemize}

\section{Notations and preliminaries}
\noindent \textbf{Notations}. In the whole paper, the notation $A \lesssim B$ means that there exists a constant $C>0$, that only depends on some given parameters, such that $0 \leq A \leq C B$. We will write $A \simeq B$ when $A\lesssim B$ and $B\lesssim A$. Typically, the given parameters will be the regularity of the solution and the value of the resistivity. These quantities are fixed throughout the proofs and the choice of the other (large or small) quantities like the frequency scale $N$ and the size of the initial data $\delta$ depends (implicitly) on them. In order to stress out the dependence of the 
constant $C$ on a parameter $\theta$ we also use the notations
$A \lesssim_{\theta} B$ and $A \simeq_{\theta} B$. 

\subsection{Functional spaces}
\textbf{Sobolev spaces}. Denote by $\mathscr{S}'$ the space of tempered distributions (the topological dual space of Schwartz functions) and consider the usual notation for partial derivatives
$$\partial^\alpha=\frac{\partial^{\abs{\alpha}}}{\partial_{x_1}^{\alpha_1}\partial_{x_2}^{\alpha_2}\partial_{x_3}^{\alpha_3}}, \quad\text{for }\alpha=(\alpha_1,\alpha_2,\alpha_3)\in\N_0^3, \quad\abs{\alpha}=\alpha_1+\alpha_2+\alpha_3.$$
We will work in the setup of Sobolev spaces
$$\forall r\in\N,\quad W^{r,p}(\R^3):=\conj{f\in\mathscr{S}'(\R^3)}{\partial^\alpha f\in L^p, \text{ for all } \abs{\alpha}\leq r}$$
with the norm
$$\norm{f}{W^{r,p}}:=\left(\sum_{m=0}^r\norm{\nabla^mf}{L^p}^p\right)^{1/p}, \quad\text{where } \abs{\nabla^mf}^p:=\sum_{\abs{\alpha}=m}\abs{\partial^\alpha f}^p.$$
When $p=2$ they are a Hilbert space $H^r(\R^3):=W^{r,2}(\R^3)$, which also has a Fourier-side definition:
$$H^s(\R^3):=\conj{f\in\mathscr{S}'(\R^3)}{\norm{f}{H^s}<\infty}, \quad\text{with}\quad \norm{f}{H^s}^2:= \int_{\R^3}(1+\abs{\xi}^2)|\hat{f}(\xi)|^2\,\d\xi \quad \forall s\geq0.$$
Both definitions coincide when $s$ in a non negative integer. We shall denote with $\dot{H}^r$ the classical homogeneous Sobolev spaces
\begin{equation*}
    \dot{H}^r := \big\{ f\in\mathscr{S}' \colon \hat{f}\in L_{\text{loc}}^1(\R^3) \text{ and } \sum_{|\alpha|=r}|\partial^\alpha f|^2 < \infty \big\},
\end{equation*}
where $\hat{f}$ stands for the Fourier transform of $f$. For the sake of simplicity, we will denote $W^{r,p} \equiv W^{r,p}(\R^3)$ and $H^s \equiv H^s(\R^3)$.

\medskip

\noindent
\textbf{Besov Spaces and Littlewood-Paley decomposition}. We use the notations of \cite{CF23} (see also Chapter 2 of \cite{BCD11} for details). Fix a smooth radial function $\chi\in\mathscr{S}$ whose Fourier transform satisfies
\begin{equation*}
    \hat{\chi}(\xi) = \begin{cases}
        1 & \text{for } |\xi|\leq1, \\
        0 & \text{for } |\xi|>2
    \end{cases}
\end{equation*}
and such that $r\mapsto\chi(re)$ is nonincreasing over $\R_+$ for all unitary vectors $e\in\R^3$. Set $\varphi(\xi):=\chi(\xi)-\chi(2\xi)$ and $\varphi_j(\xi):=\varphi(2^{-j}\xi)$ for all $j\geq0$. The dyadic blocks $(\Delta_j)_{j\in\Z}$ are defined by\footnote{Throughout we agree that if $f$ is a measurable function on $\R^n$ with at most polynomial growth at infinity, then $f(D)$ denotes the pseudo-differential operator $u\mapsto \mathcal{F}^{-1}(f\hat{u})$.}:
\begin{equation*}
    \forall j\in\Z, \quad \dot{\Delta}_j:=\varphi_j(D).
\end{equation*}
The following property holds true:
\begin{equation*}
    {\rm Id} = \sum_{j\in\Z} \dot{\Delta}_j \text{ in } \mathscr{S}_h'(\R^n),
\end{equation*}
where
\begin{defi}[Definition 1.26 in \cite{BCD11}]
    We denote by $\mathscr{S}_h'(\R^3)$ the space of tempered distributions $f$ such that
    \begin{equation*}
        \lim_{\lambda\rightarrow\infty} \norm{\theta(\lambda D)f}{L^\infty}=0 \quad \text{for any }\, \theta\in\mathcal{D},
    \end{equation*}
    being $\mathcal{D}$ the space of test functions $C_c^{\infty}$.
\end{defi}
Now we can define the homogeneous Besov spaces:
\begin{defi}\label{def:Besov.spaces}
    Let $s\in\R$ and $(p,r)\in[1,\infty]^2$. The homogeneous Besov space $\dot{B}_{p,r}^s$ consists of those distributions $f\in\mathscr{S}_h'$ such that
    \begin{equation*}
        \norm{f}{\dot{B}_{p,r}^s} := \bigg( \sum_{j\in\Z} 2^{rjs}\norm{\dot{\Delta}_jf}{L^p}^r \bigg)^{1/r} <\infty \quad\text{if } r<\infty,
    \end{equation*}
    and
    \begin{equation*}
        \norm{f}{\dot{B}_{p,\infty}^s} := \sup_{j\in\Z} (2^{js}\norm{\dot{\Delta}_jf}{L^p}).
    \end{equation*}
\end{defi}
\begin{remark}\label{rmk:Besov.caloric}
    If $s<0$, one also has the equivalent caloric formulation
    \begin{equation*}
        \norm{f}{\dot{B}_{p,r}^{s}} \sim \norm{t^{-s/2} \norm{e^{t\Delta}f}{L^p}}{L^r\big(\R_+;\,\frac{\d t}{t}\big)},
    \end{equation*}
    where $e^{t\Delta}$ denotes the heat kernel.
\end{remark}

\subsection{Harmonic and differential inequalities}
We will need some analytic tools that can be found in any Harmonic Analysis and Partial Differential Equations book (see for example \cite{Bre10,BCD11}). We also give more specific references for the reader convenience.

\begin{lem}[Grönwall's inequality, Lemma A.24 in \cite{RRS16}]
    Let $u \colon [0,T] \rightarrow [0,\infty)$ be an absolutely continuous function that satisfies the differential inequality
    \begin{equation*}
        \forall t\in[0,T], \quad u'(t) \leq a(t)u(t) + b(t),
    \end{equation*}
    where $a$ and $b$ are non-negative integrable functions. Then,
    \begin{equation*}
        \forall t \in [0,T], \quad u(t) \leq e^{A(t)} \bigg( u(0) + \int_0^t e^{-A(s)}b(s)\,\d s \bigg),
    \end{equation*}
    where $A(t) := \int_0^t a(s)\,\d s$.
\end{lem}
\begin{remark}
    We will use another version of Grönwall's inequality, which can be restated avoiding the sign hypothesis over the coefficients and requiring less integrability over $u$, as follows: let $u\in C([0,T);\R)$ for $T\in(0,\infty)$ satisfy the differential inequality
    \begin{equation*}
        \forall t\in[0,T), \quad u'(t)\leq a(t)u(t)+b(t),
    \end{equation*}
    where $a,b\in L^1(0,T)$. Then,
    \begin{equation*}
        u(t) \leq u(0)e^{A(t)} + \int_0^t b(s)e^{A(t)-A(s)}\,\d s. 
    \end{equation*}
    For the proof, define
    \begin{equation*}
        v(t) := u(t)e^{-A(t)} - \int_0^t b(s)e^{-A(s)}\,\d s,
    \end{equation*}
    which is continuous by the Fundamental Theorem of Calculus. By using mollifiers, one can check that $v'$ is well-defined and $v'\leq0$ on $(0,T)$ so, without loss of generality, one can assume that $u$ is actually absolutely continuous. Now, multiply the differential inequality by the factor $e^{-A(t)}$, so that
    \begin{equation*}
        \forall t\in[0,T), \quad u'(t)e^{-A(t)}\leq a(t)u(t)e^{-A(t)}+b(t)e^{-A(t)} \quad\iff\quad (ue^{-A})'(t) \leq b(t)e^{-A(t)}\,.
    \end{equation*}
    Integrating this last expression and multiplying by $e^{A(t)}$, one obtains the desired result:
    \begin{equation*}
        u(t) \leq u(0)e^{A(t)} + \int_0^t b(s)e^{A(t)-A(s)}\,\d s.
    \end{equation*}
\end{remark}

\begin{prop}[Interpolation of $L^p(\R^n)$ spaces, Theorem 1.5 in \cite{RRS16}]
    Let $1\leq p_0<p_1\leq\infty$ and $u\in L^{p_0}\cap L^{p_1}$. Then, $u\in L^p$ for all $p_0\leq p\leq p_1$ and
    $$\norm{u}{L^{p_\theta}} \leq \norm{u}{L^{p_0}}^{\theta} \norm{u}{L^{p_1}}^{1-\theta}, \quad \text{where } \frac{1}{p_\theta} = \frac{\theta}{p_0} + \frac{1-\theta}{p_1}.$$
\end{prop}

\begin{prop}[Interpolation of $H^s(\R^n)$ spaces, Proposition 1.52 in \cite{BCD11}]
    If $s_0\leq s\leq s_1$, then we have
    $$\norm{u}{H^s} \leq \norm{u}{H^{s_0}}^{\theta} \norm{u}{H^{s_1}}^{1-\theta} \quad\text{with } s=\theta s_0 + (1-\theta)s_1.$$
\end{prop}

\begin{prop}[Gagliardo-Nirenberg inequality \cite{Ni59}]\label{thm:Gagliardo-Nirenberg}
Let $1\leq q,r \leq \infty$ and $j$ and $m$ integers with $0\leq j<m$, the following inequalities hold:
    \begin{equation}\label{eq:Gagliardo-Nirenberg_ineq}
    	\norm{\nabla^{j}u}{L^p(\R^3)} \leq C\norm{\nabla^{m}u}{L^r(\R^3)}^\theta \norm{u}{L^q(\R^3)}^{1-\theta},
    \end{equation}
    (the constant $C$ depending only on $m,\,j,\,q,\,r,\,\theta$), where
    \begin{equation*}
    	\frac{1}{p} = \frac{j}{3} + \theta\bigg(\frac{1}{r}-\frac{m}{3}\bigg) + \frac{1 - \theta}{q},
	\qquad j/m\leq\theta\leq1
    \end{equation*}
    with the following exceptions:
    \begin{enumerate}
    	\item If $j=0$, $rm<3$ and $q=\infty$, then we make an additional assumption that either $u$ tends to zero at infinity or $u\in L^{\tilde{q}}$ for some finite $\tilde{q}>0$.
    	\item If $r > 1$ and $m-j-3/r$ is a non negative integer, then \eqref{eq:Gagliardo-Nirenberg_ineq} holds only for $\theta$ satisfying $j/m\leq\theta<1$.
    \end{enumerate}
\end{prop}
\begin{remark}\label{rmk:GN.examples}
	During the proof of Theorem \ref{thm:global_existence_large_data} we will use Proposition \ref{thm:Gagliardo-Nirenberg}  with:
	\begin{itemize}
        \item $j=0,\, m=1,\, r=2,\, q=2,\, p=3$ and $\theta=1/2$, so that
		\begin{equation*}
			\norm{u}{L^{3}} \leq C \norm{\nabla u}{L^2}^{1/2} \norm{u}{L^2}^{1/2}\,.
		\end{equation*}
       %
%
	\end{itemize}
\end{remark}

\begin{prop}[Sobolev embeddings for integer regularity in $\R^3$, Corollary 9.13 in \cite{Bre10}]
    Let $m\geq1$ an integer and $p\in[1,\infty)$. Then,
    \begin{align*}
        W^{m,p}(\R^3)\subset L^q(\R^3)&, \quad \text{where } \frac{1}{q}=\frac{1}{p}-\frac{m}{3},& \quad\text{if } \frac{1}{p}-\frac{m}{3}>0; \\
        W^{m,p}(\R^3)\subset L^q(\R^3)&, \quad \forall q\in[p,\infty),& \quad\text{if } \frac{1}{p}-\frac{m}{3}=0; \\
        W^{m,p}(\R^3)\subset L^\infty(\R^3)&,& \quad\text{if } \frac{1}{p}-\frac{m}{3}<0
    \end{align*}
    and all these injections are continuous. Moreover, if $m-3/p>0$ is not an integer, then
    $$W^{m,p}(\R^3)\subset C^k(\R^3),$$
    where $k:=[m-3/p]$.
\end{prop}

\begin{prop}[Heat kernel estimates, Theorem D.4 in \cite{RRS16}]\label{prop:heat.dispersion}
    Let $e^{t\Delta}$ be the $n$-dimensional heat kernel (that acts as a convolution over tempered distributions). Then, for all $1\leq p\leq q\leq\infty$, $\alpha\in\N_0^n$ and $t>0$ it holds that
    \begin{align*}
    	\norm{e^{t\Delta}f}{L^q} &\leq \frac{C}{t^{\frac{n}{2}\left(\frac{1}{p}-\frac{1}{q}\right)}}\norm{f}{L^p}, \\
    	\norm{\partial^\alpha e^{t\Delta}f}{L^q} &\leq \frac{C}{t^{\frac{|\alpha|}{2}+\frac{n}{2}\left(\frac{1}{p}-\frac{1}{q}\right)}}\norm{f}{L^p}.
    \end{align*}
    Moreover, for any $r\in\N$ the following bound holds:
    $$\norm{e^{t\Delta}\nabla f}{H^r} \leq \frac{C}{\sqrt{t}}\norm{f}{H^r}.$$
\end{prop}

Let $\abs{D}^k$ be the Fourier multiplier with symbol $\abs{\xi}^k$. Define the operator $P_{\leq R}$ as 
\begin{equation}\label{dfjnskfgnsk2}
\widehat{P_{\leq R}f}(\xi) := \chi(\xi/R)\hat{f}(\xi),
\end{equation}
where $\chi\geq0$ is a smooth cut-off of the ball $B_{1/2}(0)$:
\begin{equation}\label{dfmnjskgdsknj1}
    \chi(\xi):= \begin{cases}
        1 & \text{ for } \xi\in B_{1/2}(0), \\
	0 & \text{ for } \xi\in\R^3\setminus \overline{B_1(0)}.
    \end{cases}
\end{equation}
We also set 
\begin{equation}\label{dfjnskfgnsk3}
P_{>R}:=1-P_{\leq R}.
\end{equation}
\begin{prop}[Bernstein inequalities]\label{prop:Bernstein_inequalities}
	Let $R >0$, $k \geq 0$ and $1\leq p\leq q\leq\infty$. Then, for any $f\in L^p(\R^3)$ the following estimates hold:
	\begin{align*}
		\norm{P_{\leq R}f}{L^q} &\lesssim R^{3\left(\frac{1}{p} - \frac{1}{q}\right)}\norm{f}{L^p}, \\
		\norm{\nabla P_{\leq R}f}{L^p} &\lesssim R\norm{P_{\leq R}f}{L^p} \\
		\norm{P_{>R}f}{L^p} &\lesssim R^{-k}\norm{P_{>R}\abs{D}^kf}{L^p}.
	\end{align*}
\end{prop}

The following result describes the action of the heat semigroup on distributions with Fourier transform supported in an annulus.
\begin{prop}[\cite{BCD11}, Lemma 2.4]\label{prop:heat.act.annulus}
	Let $\mathcal{A}\subset\R^3$ be an annulus. Positive constants $c$ and $C$ exist such that for any $p$ in $[1,\infty]$ and any couple $(t,\lambda)$ of positive real numbers, we have
	$$\text{supp}\,\hat{u}\subset\lambda\mathcal{A} \Rightarrow \norm{e^{t\Delta}u}{L^p}\leq Ce^{-ct\lambda^2}\norm{u}{L^p}.$$
\end{prop}

\subsection{Beltrami fields}
A vector field $\fun{B_\lambda}{\R^3}{\R^3}$ is a (strong) \textit{Beltrami field} with frequency $\lambda\in\R\setminus\{0\}$ if it is an eigenfunction of the curl operator with eigenvalue $\lambda$, that is,
$$\nabla\times B_\lambda=\lambda B_\lambda.$$
These vector fields are divergence free and have zero mean:
$$\div B_\lambda=0 \quad\text{and}\quad \int_{\R^3}B_\lambda\,\d x=0.$$
They additionally are eigenvectors of the Laplacian and steady solutions to the Euler equation:
$$\Delta B_\lambda=-\lambda^2B_\lambda, \quad (B_\lambda\cdot\nabla)B_\lambda=-\nabla\frac{|B_\lambda|^2}{2}.$$
With this in mind, one can check that $(0,e^{-\eta t\lambda^2}B_\lambda)$ is a solution of \eqref{eq:MHD} with initial data $(0,B_\lambda)$, for a suitable choice of the  pressure. Given a Beltrami field $B_{\lambda}$, the rescaled vector field $a B_{\lambda}$
is still a Beltrami field (with frequency $\lambda$), for all $a \in \R$. Thus 
$(0,e^{-\eta t  \lambda^2} a B_\lambda)$ is still a solution of \eqref{eq:MHD} (with appropriate choice of the pressure). 
This rescaling shows that once we fix a normed space such that $\|B_{\lambda}\| < + \infty$, 
we can construct 
global smooth solutions solutions with initial data of any assigned norm, making vary $\alpha \in \R$. 
It is worth noting that $\|B_{\lambda}\|_{L^2(\R^3)} = + \infty$ for any Beltrami field.

\section{Global existence result}
\noindent
We prove global existence and uniqueness of a (unique) strong solution starting from initial data of arbitrary size in critical spaces, satisfying some smallness condition on the difference between the initial velocity and magnetic fields. The result lies in the regularizing effect of the interactions between the velocity and the magnetic fields due to cancellations. We restate the result here for the reader's convenience:
\setcounter{thm2}{0}
\begin{thm2}\label{thm:global_existence_large_data}
Let $r\geq3$. There exists  $c >0$ sufficiently small that the following holds.
	Let $u_0,b_0\in H^r(\R^3)$ be such that
	\begin{equation}\label{eq:hypothesis_data_for_global_existence}
		\norm{u_0-b_0}{H^r} \leq \rho \quad\text{and}\quad \norm{e^{\eta t\Delta}u_0}{L_t^2W_x^{k,\infty}} + \norm{e^{\eta t\Delta}b_0}{L_t^2W_x^{k,\infty}}  \lesssim 1+N^k, \quad \forall \quad 0\leq k\leq r,
	\end{equation}
	with $0<\rho\leq cN^{-r-1}$. Then, there exists a unique global regular solution $(u,b)$ to \eqref{eq:MHD} with $\nu = \eta$ that satisfies
	$$\norm{u(t)-e^{\eta t \Delta}u_0}{H^k} + \norm{b(t)-e^{\eta t\Delta}b_0}{H^k} \leq C\rho N^k, \quad\forall 0\leq k\leq r, \quad\forall t\geq0.$$
\end{thm2}
\begin{proof}
	Let $( u, b)$ be a local strong solution to \eqref{eq:MHD} with $\eta = \nu$, whose existence is guaranteed by \cite{DL72}. Consider 
	\begin{equation}\label{eq:def_perturbations}
		v(t):=u(t)-e^{\eta t\Delta} u_0 \quad\text{and}\quad h(t):=b(t)-e^{\eta t\Delta}b_0.
	\end{equation}
One can immediately check that $(v,h)$ satisfy the equations
	\begin{equation}\label{fmdnjskgnjsgv}
		\left\lbrace\begin{aligned}
			\partial_t v - \eta \Delta v+(v \cdot\nabla)v-(h \cdot \nabla)h &= -\nabla p - \left(e^{\eta t \Delta}u_0\cdot\nabla\right)e^{\eta t \Delta}u_0 + \left(e^{\eta t \Delta}b_0\cdot\nabla\right)e^{\eta t \Delta}b_0 - \left(e^{\eta t \Delta}u_0\cdot\nabla\right)v \\
			&\quad +\left(e^{\eta t \Delta}b_0\cdot\nabla\right)h -\left(v\cdot\nabla\right)e^{\eta t \Delta}u_0 
			+ \left(h\cdot\nabla\right)e^{\eta t\Delta}b_0, \\
			\partial_t h- \eta \Delta h+(v\cdot\nabla)h-(h\cdot\nabla)v 
			&= -\left(e^{\eta t \Delta}u_0\cdot\nabla\right)e^{\eta t\Delta}b_0 
			+ \left(e^{\eta t \Delta}b_0\cdot\nabla\right)e^{\eta t \Delta}u_0 - \left(e^{\eta t \Delta}u_0\cdot\nabla\right)h \\
			&\quad +\left(e^{\eta t \Delta}b_0\cdot\nabla\right)v 
			- \left(v\cdot\nabla\right)e^{\eta t \Delta}b_0 + \left(h\cdot\nabla\right)e^{\eta t \Delta}u_0, \\
			\div v=\div h=0, \hspace{2cm} \\
			v(0,\cdot)=0=h(0,\cdot). \hspace{1.7cm}
		\end{aligned}\right.
	\end{equation}	
	For each $0\leq k\leq r$, denote the energy of level $k$ by
	$$e_k(t):=\sum_{\abs{\alpha}\leq k}\int_{\R^3}\left(\abs{\partial^\alpha v(x)}^2 + \abs{\partial^\alpha h(x)}^2\right)\,\d x.$$
	We are going to carry out an induction process on the energy levels. The first thing that we do is to rewrite the following terms of the equations for $v$:
	\begin{equation}\label{eq:rewrite.lin.terms.level0v}
		-\left(e^{\eta t \Delta}u_0\cdot\nabla\right)e^{\eta t \Delta}u_0 
		+ \left(e^{\eta t \Delta}b_0\cdot\nabla\right)e^{\eta t \Delta}b_0 
		= \left(e^{\eta t \Delta}u_0\cdot\nabla\right)e^{\eta t \Delta}(b_0-u_0) 
		+ \left(e^{\eta t \Delta}(b_0-u_0)\cdot\nabla\right)e^{\eta t \Delta}b_0
	\end{equation}
	and the same in the equations for $h$:
	\begin{equation}\label{eq:rewrite.lin.terms.level0h}
		-\left(e^{\eta t \Delta}u_0\cdot\nabla\right)e^{\eta t \Delta}b_0 
		+ \left(e^{\eta t \Delta}b_0\cdot\nabla\right)e^{\eta t \Delta}u_0 
		= \left(e^{\eta t \Delta}(b_0-u_0)\cdot\nabla\right)e^{\eta t \Delta}b_0 
		+ \left(e^{\eta t \Delta}b_0\cdot\nabla\right)e^{\eta t \Delta}(u_0-b_0).
	\end{equation}
	For the zero energy level $k=0$, take the scalar product of the equations by $v$ and $h$, respectively, integrate them over space and add them:
	\begin{equation*}
		\begin{split}
			\frac{1}{2}\frac{\d}{\d t}\int_{\R^3} \Bigg(|v(t)|^2&+|h(t)|^2\Bigg)\,\d x 
			+ \eta \int_{\R^3}\Bigg(|\nabla v(t)|^2+|\nabla h(t)|^2\Bigg)\,\d x \\
			&= \int_{\R^3} \left(e^{\eta t \Delta}u_0\cdot\nabla\right)e^{\eta t \Delta}(b_0-u_0)\cdot v \,\d x \quad+ \int_{\R^3} \left(e^{\eta t \Delta}(b_0-u_0)\cdot\nabla\right)e^{\eta t \Delta}b_0\cdot v\,\d x \\
			&\quad+ \int_{\R^3} \left(e^{\eta t \Delta}(b_0-u_0)\cdot\nabla\right)e^{\eta t \Delta}b_0\cdot h\,\d x + \int_{\R^3} \left(e^{\eta t \Delta}b_0\cdot\nabla\right)e^{\eta t \Delta}(u_0-b_0)\cdot h\,\d x \\
			&\quad+ \int_{\R^3}  \left(e^{\eta t \Delta}b_0\cdot\nabla\right)h\cdot v\,\d x + \int_{\R^3} \left(e^{\eta t \Delta}b_0\cdot\nabla\right)v\cdot h\,\d x \\
			&\quad- \int_{\R^3} \left(v\cdot\nabla\right)e^{\eta t \Delta}u_0\cdot v\,\d x + \int_{\R^3} \left(h\cdot\nabla\right)e^{\eta t \Delta}u_0\cdot h\,\d x \\
			&\quad+ \int_{\R^3}\left(h\cdot\nabla\right)e^{\eta t \Delta}b_0\cdot v\,\d x - \int_{\R^3} \left(v\cdot\nabla\right)e^{\eta t \Delta}b_0\cdot h\,\d x.
		\end{split}
	\end{equation*}
	First of all, notice that, since we have taken the scalar product of the  
	the first equation against  a divergence-free 
	vector field and then integrated over $\R^3$, the pressure term contribution 
	has vanished, as it can be seen after integration by parts. 
	We will also use the fact that $u_0, b_0, v, h$ are divergence free,
	and that 
	 integration by parts of the nabla into expressions like $\int_{\R^3} (f \cdot \nabla) g \cdot h$ with $\div f=0$ gives 
	$$
	\int_{\mathbb{R}^3} (f \cdot \nabla) g \cdot h = - \int_{\mathbb{R}^3} (f \cdot \nabla) h \cdot g. 
	$$
	
	 Now, each integral will be bounded separately. In each line, both integrals that appear are bounded in the same way, so we write just one of them. In all of them, one has to integrate by parts first and then use Hölder's inequality, the continuity of the heat semigroup in $L^2$ and $\epsilon$-Young's inequality, combined with the hypothesis \eqref{eq:hypothesis_data_for_global_existence}. The constant $C$ on the right hand side is $\varepsilon$ dependent, however we will not keep track of this in the notations.
		For the first line:
	\begin{equation*}
		\begin{split}
			\abs{\int_{\R^3} \left(e^{\eta t \Delta}u_0\cdot\nabla\right)e^{\eta t \Delta}(b_0-u_0)\cdot v \,\d x} &\leq \norm{\nabla v}{L^2} \norm{e^{\eta t \Delta}u_0}{L^\infty} \norm{e^{\eta t \Delta}(b_0-u_0)}{L^2} \\
			&\leq \norm{\nabla v}{L^2} \norm{e^{\eta t \Delta}u_0}{L^\infty} \norm{b_0-u_0}{L^2} \\
			&\leq \epsilon\norm{\nabla v}{L^2}^2 + C\rho^2\norm{e^{\eta t \Delta}u_0}{L^\infty}^2,
		\end{split}
	\end{equation*}
	where in the integration by parts we have used the fact that $e^{\eta t \Delta}u_0$ is divergence-free. For the second line:
	\begin{equation*}
		\begin{split}
			\int_{\R^3} \left(e^{\eta t \Delta}(b_0-u_0)\cdot\nabla\right)e^{\eta t \Delta}b_0\cdot h\,\d x &\leq \norm{\nabla h}{L^2} \norm{e^{\eta t \Delta}(b_0-u_0)}{L^2} \norm{e^{\eta t \Delta}b_0}{L^\infty} \\
			&\leq \epsilon\norm{\nabla h}{L^2}^2 + C\rho^2\norm{e^{\eta t \Delta}b_0}{L^\infty}^2.
		\end{split}
	\end{equation*}
	The integrals in the third line cancel each other just by integrating by parts. Finally, the integrals in the last two lines are bounded as follows:
	\begin{equation}\label{eq:Gronwall_exp_r0}
		\begin{split}
			\abs{\int_{\R^3} \left(v\cdot\nabla\right)e^{\eta t \Delta}u_0\cdot v\,\d x} &\leq \norm{\nabla v}{L^2} \norm{v}{L^2} \norm{e^{t\eta\Delta}u_0}{L^\infty} \\
			&\leq \epsilon\norm{\nabla v}{L^2}^2 + C\norm{e^{\eta t \Delta}u_0}{L^\infty}^2\norm{v}{L^2}^2.
		\end{split}
	\end{equation}
	Put all the inequalities together: choosing $\epsilon=\eta/16$, the gradient terms can be absorbed by the diffusion on the left-hand side. We end up with
	\begin{multline*}
		\frac{\d}{\d t}e_0(t) + \eta \int_{\R^3}\Bigg(|\nabla v(t)|^2+|\nabla h(t)|^2\Bigg)\,\d x \\
		\leq C\left(\norm{e^{\eta t \Delta}u_0}{L^\infty}^2 + \norm{e^{\eta t \Delta}b_0}{L^\infty}^2\right) e_0(t) + C\rho^2\left(\norm{e^{\eta t \Delta}u_0}{L^\infty}^2 + \norm{e^{\eta t \Delta}b_0}{L^\infty}^2\right).
	\end{multline*}
	Applying Grönwall's inequality and noticing that $e_0(0)=0$,
	$$e_0(t) \leq C\int_0^t \rho^2\left(\norm{e^{\eta s  \Delta}u_0}{L^\infty}^2 
	+ \norm{e^{\eta s  \Delta}b_0}{L^\infty}^2\right) e^{C\int_s^t \left(\norm{e^{ \eta \tau  \Delta}u_0}{L^\infty}^2 
	+ \norm{e^{\eta \tau \Delta}b_0}{L^\infty}^2\right)\,\d\tau}\,\d s \leq C\rho^2,$$
	where we have used the hypothesis \eqref{eq:hypothesis_data_for_global_existence} with $k=0$.
	
	\medskip
	
	For the induction step, suppose that the following estimates hold:
	\begin{equation}\label{eq:induction_hypothesis}
		e_k(t) \leq C\rho^2N^{2k}, \quad \forall t\geq0,\; \forall1\leq k\leq r-1.
	\end{equation}
	We must prove 
	\begin{equation}\label{eq:induction_hypothesis}
		e_r(t) \leq C\rho^2N^{2r}, \quad \forall t\geq0.
	\end{equation}
     To do so we project the equations for $(v,h)$ outside of the (Fourier) ball of 
	radius using the smooth frequency cut-off defined in \eqref{dfmnjskgdsknj1}. 
	We thus define
	\begin{equation}\label{eq:notation.proj}
	f^{\leq 1} := P_{\leq 1}f \quad\text{and}\quad f^{>1}:=P_{>1}f.
\end{equation}
where $P_{\leq 1}, P_{>1}$ are defined in \eqref{dfjnskfgnsk2}-\eqref{dfjnskfgnsk3}, and
    \begin{equation*}
        \pib{e_r}(t) := \sum_{\abs{\alpha}\leq r} \int_{\R^3} \bigg(\abs{\partial^\alpha \pib{v}}^2 + \abs{\partial^\alpha \pib{h}}^2\bigg)\d x, \quad \pob{e_r}(t) := \sum_{\abs{\alpha}\leq r} \int_{\R^3} \bigg(\abs{\partial^\alpha \pob{v}}^2 + \abs{\partial^\alpha \pob{h}}^2\bigg)\d x.
    \end{equation*}
    We note that 
    \begin{equation}\label{BPIN}
       \|\nabla f^{>1} \|_{L^2} \geq \frac12 \| f^{>1} \|_{L^2}
    \end{equation}
    and that by Bernstein inequalities \eqref{prop:Bernstein_inequalities} we have for all  multi-index $\alpha$:
    \begin{equation}\label{BPIN2}
      \| \partial^{\alpha} f^{\leq 1} \|_{L^p} \leq C_{\alpha} \| f^{\leq 1} \|_{L^2} \leq C_{\alpha} \| f \|_{L^2}, \qquad p \in [2, \infty], 
    \end{equation}
   and in particular:
    \begin{equation}\label{BPIN3}
      \|  f^{\leq 1} \|_{L^p} \leq C \| f^{\leq 1} \|_{L^2} \leq C \| f \|_{L^2}, \qquad p \in [2, \infty], 
    \end{equation}

    For $1\leq\abs{\alpha}\leq r$, apply $\partial^\alpha$ to both equations, multiply each one by $\partial^\alpha\pob{v}$ and $\partial^\alpha\pob{h}$, respectively, and use Leibnitz product's rule for the derivatives of the products. Finally, integrate the equations over space and use that $P_{>1}$ is a Fourier multiplier (self-adjoint operator in $L^2$) and a projection ($P_{>1}\pob{v}=\pob{v}$) to obtain the following:
	$$\frac{1}{2}\frac{\d}{\d t} \int_{\R^3} \bigg(\abs{\partial^\alpha \pob{v}}^2 +  \abs{\partial^\alpha \pob{h}}^2\bigg)\d x + \eta \int_{\R^3} \left(\abs{\nabla\partial^\alpha \pob{v}}^2 + \abs{\nabla\partial^\alpha \pob{h}}^2\right)\d x= \sum_{\beta\leq\alpha}\binom{\alpha}{\beta}(I_1+I_2+I_3),$$
	where we separate the fully nonlinear term
	\begin{align*}
		I_1 &:= -\int_{\R^3} \left(\partial^\beta v\cdot\nabla\right)\partial^{\alpha-\beta}v\cdot\partial^\alpha \pob{v}\,\d x + \int_{\R^3} \left(\partial^\beta h\cdot\nabla\right)\partial^{\alpha-\beta}h\cdot\partial^\alpha \pob{v}\,\d x \\
		&\quad- \int_{\R^3} \left(\partial^\beta v\cdot\nabla\right)\partial^{\alpha-\beta}h\cdot\partial^\alpha \pob{h}\,\d x + \int_{\R^3} \left(\partial^\beta h\cdot\nabla\right)\partial^{\alpha-\beta}v\cdot\partial^\alpha \pob{h}\,\d x,
	\end{align*}
	the linear term
    \begin{equation}\label{eq:lin.term.levelk}
        \begin{split}
            I_2 &:= \int_{\R^3} \left(e^{\eta t \Delta}\partial^\beta u_0\cdot\nabla\right)e^{\eta t \Delta}\partial^{\alpha-\beta}(b_0-u_0)\cdot\partial^\alpha \pob{v} \,\d x \\
    		&\quad+ \int_{\R^3} \left( e^{\eta t\Delta}\partial^\beta(b_0-u_0)\cdot\nabla\right)e^{\eta t \Delta}\partial^{\alpha-\beta}b_0\cdot\partial^\alpha \pob{v} \,\d x \\
    		&\quad+ \int_{\R^3} \left( e^{\eta t \Delta}\partial^\beta(b_0-u_0)\cdot\nabla\right)e^{\eta t \Delta}\partial^{\alpha-\beta}b_0\cdot \partial^\alpha \pob{h}\,\d x \\
    		&\quad+ \int_{\R^3} \left(e^{\eta t \Delta}\partial^\beta b_0\cdot\nabla\right)e^{\eta t\Delta}\partial^{\alpha-\beta}(u_0-b_0)\cdot \partial^\alpha \pob{h}\,\d x
        \end{split}
    \end{equation}
	and the mixed term
	\begin{align*}
		I_3 &:= -\int_{\R^3} \left(e^{\eta t \Delta}\partial^{\beta}u_0\cdot\nabla\right)\partial^{\alpha-\beta}v\cdot\partial^\alpha\pob{v}\,\d x - \int_{\R^3} \left(e^{\eta t \Delta}\partial^{\beta}u_0\cdot\nabla\right)\partial^{\alpha-\beta}h\cdot\partial^\alpha\pob{h}\,\d x \\
		&\quad +\int_{\R^3} \left(e^{\eta t \Delta}\partial^\beta b_0\cdot\nabla\right)\partial^{\alpha-\beta}h\cdot \partial^\alpha \pob{v}\,\d x + \int_{\R^3} \left(e^{\eta t \Delta}\partial^\beta b_0\cdot\nabla\right)\partial^{\alpha-\beta}v\cdot \partial^\alpha \pob{h}\,\d x \\
		&\quad - \int_{\R^3} \left(\partial^\beta v\cdot\nabla\right)e^{\eta t \Delta}\partial^{\alpha-\beta}u_0\cdot \partial^\alpha \pob{v}\,\d x + \int_{\R^3} \left(\partial^\beta h\cdot\nabla\right)e^{\eta t \Delta}\partial^{\alpha-\beta}u_0\cdot \partial^\alpha \pob{h}\,\d x \\
		&\quad + \int_{\R^3}\left(\partial^\beta h\cdot\nabla\right)e^{\eta t \Delta}\partial^{\alpha-\beta}b_0\cdot \partial^\alpha \pob{v}\,\d x - \int_{\R^3} \left(\partial^\beta v\cdot\nabla\right)e^{\eta t \Delta}\partial^{\alpha-\beta}b_0\cdot \partial^\alpha \pob{h}\,\d x.
	\end{align*}
	The following step is to bound each integral. We first focus on $I_1$. We separate the terms with $|\beta|=0$ from the rest:
	\begin{align*}
		\sum_{\beta\leq\alpha} \int_{\R^3} \left(\partial^\beta v\cdot\nabla\right)\partial^{\alpha-\beta}v\cdot\partial^\alpha \pob{v} &= \int_{\R^3}(v\cdot\nabla)\partial^\alpha v \cdot\partial^\alpha \pob{v} + \sum_{0<\beta\leq\alpha} \int_{\R^3} \left(\partial^\beta v\cdot\nabla\right)\partial^{\alpha-\beta}v\cdot\partial^\alpha \pob{v}, \\
		\sum_{\beta\leq\alpha} \int_{\R^3} \left(\partial^\beta h\cdot\nabla\right)\partial^{\alpha-\beta}h\cdot\partial^\alpha \pob{v} &= \int_{\R^3} \left(h\cdot\nabla\right)\partial^{\alpha}h\cdot\partial^\alpha \pob{v} + \sum_{0<\beta\leq\alpha} \int_{\R^3} \left(\partial^\beta h\cdot\nabla\right)\partial^{\alpha-\beta}h\cdot\partial^\alpha \pob{v}, \\
		\sum_{\beta\leq\alpha} \int_{\R^3} \left(\partial^\beta v\cdot\nabla\right)\partial^{\alpha-\beta}h\cdot\partial^\alpha \pob{h} &= \int_{\R^3} \left(v\cdot\nabla\right)\partial^{\alpha}h\cdot\partial^\alpha \pob{h} + \sum_{0<\beta\leq\alpha} \int_{\R^3} \left(\partial^\beta v\cdot\nabla\right)\partial^{\alpha-\beta}h\cdot\partial^\alpha \pob{h}, \\
		\sum_{\beta\leq\alpha} \int_{\R^3} \left(\partial^\beta h\cdot\nabla\right)\partial^{\alpha-\beta}v\cdot\partial^\alpha \pob{h} &= \int_{\R^3} \left(h\cdot\nabla\right)\partial^{\alpha} v \cdot\partial^\alpha \pob{h} + \sum_{0<\beta\leq\alpha} \int_{\R^3} \left(\partial^\beta h\cdot\nabla\right)\partial^{\alpha-\beta}v\cdot\partial^\alpha \pob{h},
	\end{align*}
	and for these terms with $|\beta|=0$ is zero we write
	\begin{equation}\label{eq:split.int.beta=0}
		\int_{\R^3}(v\cdot\nabla)\partial^\alpha v \cdot\partial^\alpha \pob{v} = \int_{\R^3}(v\cdot\nabla)\partial^\alpha \pib{v} \cdot\partial^\alpha \pob{v} + \int_{\R^3}(v\cdot\nabla)\partial^\alpha \pob{v} \cdot\partial^\alpha \pob{v}
	\end{equation}
	(the same for the other integrals with $|\beta|=0$). Now, integrating by parts and recalling that $v$ is 
	 divergence-free we notice that the the second integral on the right-hand side of \eqref{eq:split.int.beta=0} is zero. 
	In order to estimate  the first term on 
	the the right hand we use inequality \eqref{BPIN2}: 
	\begin{equation*}
		\begin{split}
			\abs{\int_{\R^3}(v\cdot\nabla)\partial^\alpha \pib{v} \cdot\partial^\alpha \pob{v}} 
			&\lesssim \norm{v}{L^2} \| \pib{v} \|_{L^2} \norm{\nabla \partial^{\alpha}\pob{v}}{L^2} \\
			&\lesssim \norm{v}{L^2}^2 \norm{\nabla\partial^{\alpha}\pob{v}}{L^2} \\
			&\lesssim \epsilon\norm{\nabla\partial^{\alpha}\pob{v}}{L^2}^2 + C\rho^4;
		\end{split}
	\end{equation*}
	note that the first inequality is justified after integrating by parts the nabla, 
	while in the last inequality we have used the estimate $\|v(t)\|^2_{L^2} \leq e_0(t)\lesssim \rho^2$ and Cauchy-Schwartz. 
	A similar estimate applies to  the rest of the integrals with $|\beta|=0$ in $I_1$. 
	
	The fully nonlinear terms with $|\beta|>0$ will be handled again splitting $v :=  \pob{v} + \pib{v}$ as well. 
	We begin with the first integral, separating
	\begin{equation*}
		\begin{split}
			\int_{\R^3} \left(\partial^\beta v\cdot\nabla\right)\partial^{\alpha-\beta}v\cdot\partial^\alpha \pob{v}\,\d x &= \int_{\R^3} \left(\partial^\beta \pib{v}\cdot\nabla\right)\partial^{\alpha-\beta}\pib{v}\cdot\partial^\alpha \pob{v}\,\d x \\
			&\quad+ \int_{\R^3} \left(\partial^\beta \pib{v}\cdot\nabla\right)\partial^{\alpha-\beta}\pob{v}\cdot\partial^\alpha \pob{v}\,\d x \\
			&\quad+ \int_{\R^3} \left(\partial^\beta \pob{v}\cdot\nabla\right)\partial^{\alpha-\beta}\pib{v}\cdot\partial^\alpha \pob{v}\,\d x \\
			&\quad+ \int_{\R^3} \left(\partial^\beta \pob{v}\cdot\nabla\right)\partial^{\alpha-\beta}\pob{v}\cdot\partial^\alpha \pob{v}\,\d x.
		\end{split}
	\end{equation*}
	We estimate term by term. The first one is bounded as follows:
	\begin{equation*}
		\begin{split}
			\abs{\int \left(\partial^\beta \pib{v}\cdot\nabla\right)\partial^{\alpha-\beta}\pib{v}\cdot\partial^\alpha \pob{v}} 
			&\leq \norm{\nabla\partial^\alpha\pob{v}}{L^2} \norm{\partial^\beta\pib{v}}{L^4} 
			\norm{\partial^{\alpha-\beta}\pib{v}}{L^4} \\
			&\lesssim \norm{\nabla\partial^\alpha\pob{v}}{L^2} \| v \|_{L^2}^2 \\
			&\leq \epsilon\norm{\nabla\partial^\alpha\pob{v}}{L^2}^2 + C\norm{v}{L^2}^4 \\
			&\leq \epsilon\norm{\nabla\partial^\alpha\pob{v}}{L^2}^2 + C\rho^4\,,
		\end{split}
	\end{equation*}
where the first inequality is justified after integrating by parts the nabla (that we will do several times in the following), then	
we have used \eqref{BPIN2} and the estimate $\|v(t)\|^2_{L^2} \leq e_0(t)\lesssim\rho^2$.

We bound the second terms as:
	\begin{equation*}
		\begin{split}
			\abs{\int \left(\partial^\beta \pib{v}\cdot\nabla\right)\partial^{\alpha-\beta}\pob{v}\cdot\partial^\alpha \pob{v}} 
			&\leq \norm{\nabla\partial^\alpha\pob{v}}{L^2} \norm{\partial^\beta\pib{v}}{L^\infty} \norm{\partial^{\alpha-\beta}\pob{v}}{L^2} \\
			&\lesssim \norm{\nabla\partial^\alpha\pob{v}}{L^2} \|v\|_{L^2} \norm{\partial^{\alpha-\beta}\pob{v}}{L^2} \\
			&\leq \epsilon\norm{\nabla\partial^\alpha\pob{v}}{L^2}^2 
		+ C\norm{v}{L^2}^2\norm{\partial^{\alpha-\beta}\pob{v}}{L^2}^2  \\
			&\leq \epsilon\norm{\nabla\partial^\alpha\pob{v}}{L^2}^2 + C\rho^2\pob{e}_r(t).
		\end{split}
	\end{equation*}
	where we have used \eqref{BPIN2} and the estimate $e_0(t)\lesssim\rho^2$.
	
	For the third one
    \begin{equation*}
        \begin{split}
            \abs{\int \left(\partial^\beta \pob{v}\cdot\nabla\right) \partial^{\alpha - \beta}\pib{v}\cdot\partial^\alpha \pob{v}} &\leq \norm{\nabla\partial^\alpha\pob{v}}{L^2} \norm{\partial^\beta \pob{v}}{L^2}\norm{\partial^{\alpha - \beta}\pib{v}}{L^\infty} \\
            &\lesssim  \norm{\nabla\partial^\alpha\pob{v}}{L^2}  \norm{\partial^{\beta}\pob{v}}{L^2} \| v \|_{L^2} \\
            &\lesssim \epsilon\norm{\nabla\partial^\alpha\pob{v}}{L^2}^2 + C\rho^2\pob{e_r}(t).
        \end{split}
    \end{equation*}

	For the fourth therm we note that, since $|\beta| \neq 0$, 
we have $|\alpha - \beta| < |\alpha|$. We then estimate
	\begin{equation}\label{Corerenjwk}
		\begin{split}
			\abs{\int \left(\partial^\beta \pob{v}\cdot\nabla\right)\partial^{\alpha-\beta}\pob{v}\cdot\partial^\alpha \pob{v}} &\leq \norm{\nabla\partial^\alpha\pob{v}}{L^2} \norm{\partial^\beta\pob{v}}{L^3} \norm{\partial^{\alpha-\beta}\pob{v}}{L^6} \\
			&\lesssim \norm{\nabla\partial^\alpha\pob{v}}{L^2} \norm{\nabla\partial^\beta\pob{v}}{L^2}^{1/2}\norm{\partial^\beta\pob{v}}{L^2}^{1/2} \norm{\nabla\partial^{\alpha-\beta}\pob{v}}{L^2} \\
						&\lesssim  
	 \Big(\sum_{\gamma: \, 0 \leq |\gamma| \leq |\alpha| } \norm{\nabla\partial^\gamma\pob{v}}{L^2}^{\frac32} \Big)
				\norm{\partial^\beta\pob{v}}{L^2}^{1/2} 
				\norm{\nabla\partial^{\alpha-\beta}\pob{v}}{L^2} \\  
		&\lesssim  
	 \Big( \sum_{\gamma: \, 0 \leq |\gamma| \leq |\alpha| } \norm{\nabla\partial^\gamma\pob{v}}{L^2}^{\frac32} \Big) \pob{e}_r(t)^{\frac34}
\\
		&\leq \varepsilon \Big(  \sum_{\gamma: \, 0 \leq |\gamma| \leq |\alpha| }  \norm{\nabla\partial^\gamma\pob{v}}{L^2}^{2} \Big) 
			+C\pob{e}_r(t)^3,
		\end{split}
	\end{equation}
where  we have used the Gagliardo-Nirenberg inequality as in Remark \ref{rmk:GN.examples} 
and the Sobolev embedding $\Dot{H}^1\hookrightarrow L^6$ and in the last inequality we have used the $\varepsilon$-Young inequality in the form $|a| |b| \leq \varepsilon |a|^{\frac43} + C_{\varepsilon} |b|^4 $.
    The exact same bounds, with obvious modifications, works for the other integrals in $I_1$, leading to the set of estimates:
	\begin{align*}
		\abs{\int \left(\partial^\beta v\cdot\nabla\right)\partial^{\alpha-\beta}v\cdot\partial^\alpha \pob{v}} 
		&
		\leq \varepsilon \sum_{\gamma: \, 0 \leq |\gamma| \leq |\alpha| } \norm{\nabla\partial^\gamma\pob{v}}{L^2}^{2} + C\rho^4 + C\rho^2\pob{e}_r(t) + C\pob{e}_r(t)^3, \\
		\abs{\int \left(\partial^\beta h\cdot\nabla\right)\partial^{\alpha-\beta}h\cdot\partial^\alpha \pob{v}} 
		&\leq  \varepsilon \sum_{\gamma: \, 0 \leq |\gamma| \leq |\alpha| } \norm{\nabla\partial^\gamma\pob{v}}{L^2}^{2} 
		+ C\rho^4 + C\rho^2\pob{e}_r(t) + C\pob{e}_r(t)^3, \\
		\abs{\int \left(\partial^\beta v\cdot\nabla\right)\partial^{\alpha-\beta}h\cdot\partial^\alpha \pob{h}} 
		&\leq 
    \varepsilon \sum_{\gamma: \, 0 \leq |\gamma| \leq |\alpha| } \norm{\nabla\partial^\gamma\pob{h}}{L^2}^{2} + C\rho^4 + C\rho^2\pob{e}_r(t) + C\pob{e}_r(t)^3, \\
		\abs{\int \left(\partial^\beta h\cdot\nabla\right)\partial^{\alpha-\beta}v\cdot\partial^\alpha \pob{h}} &\leq 
 \varepsilon \sum_{\gamma: \, 0 \leq |\gamma| \leq |\alpha| } \norm{\nabla\partial^\gamma\pob{h}}{L^2}^{2} + C\rho^4 + C\rho^2\pob{e}_r(t) + C\pob{e}_r(t)^3.
    \end{align*}

    \medskip
    
    The first integral in $I_2$, that is linear in $v,h$, is bounded as follows:
	\begin{equation*}
		\begin{split}
			\Bigg| \int_{\R^3} \left(e^{\eta t \Delta}\partial^\beta u_0\cdot\nabla\right)e^{\eta t \Delta}\partial^{\alpha-\beta}(b_0-u_0)&\cdot\partial^\alpha \pob{v} \,\d x \Bigg| \\
			&\leq \norm{\nabla\partial^\alpha \pob{v}}{L^2} \norm{e^{\eta t \Delta}\partial^\beta u_0}{L^\infty} \norm{e^{\eta t\Delta}\partial^{\alpha-\beta}(b_0-u_0)}{L^2} \\
			&\leq \epsilon\norm{\nabla\partial^\alpha \pob{v}}{L^2}^2 + C\rho^2\norm{e^{\eta t\Delta}\partial^\beta u_0}{L^\infty}^2,
		\end{split}
	\end{equation*}
where we have integrated by part the nabla in order to justify the first inequality and we have used 
$$
\norm{u_0-b_0}{H^r} \leq \rho
$$ 
and Cauchy-Schwartz in the second inequality. The exact same argument, with obvious modifications, leads to
$$
\left| \int_{\R^3} \left( e^{\eta t\Delta}\partial^\beta(b_0-u_0)\cdot\nabla\right)e^{\eta t \Delta}\partial^{\alpha-\beta}b_0\cdot\partial^\alpha \pob{v} \,\d x \right| \leq  \epsilon\norm{\nabla\partial^\alpha \pob{v}}{L^2}^2 + C\rho^2\norm{e^{\eta t\Delta}\partial^{\alpha - \beta} b_0}{L^\infty}^2,
$$
$$
\left| \int_{\R^3} \left( e^{\eta t \Delta}\partial^\beta(b_0-u_0)\cdot\nabla\right)e^{\eta t \Delta}\partial^{\alpha-\beta}b_0\cdot \partial^\alpha \pob{h}\,\d x \right| \leq  \epsilon\norm{\nabla\partial^\alpha \pob{h}}{L^2}^2 + C\rho^2\norm{e^{\eta t\Delta}\partial^{\alpha - \beta} b_0}{L^\infty}^2,
$$
$$
\left| \int_{\R^3} \left(e^{\eta t \Delta}\partial^\beta b_0\cdot\nabla\right)e^{\eta t\Delta}\partial^{\alpha-\beta}(u_0-b_0)\cdot \partial^\alpha \pob{h}\,\d x \right| \leq \epsilon\norm{\nabla\partial^\alpha \pob{h}}{L^2}^2 + C\rho^2\norm{e^{\eta t\Delta}\partial^{\beta} b_0}{L^\infty}^2
$$

    \medskip
    
	The integrals in $I_3$ are no more linear in $v$ and $h$. At this point, it is crucial to observe the following: when all the derivatives hit the nonlinear terms, there appear no extra derivatives hitting the terms involving the evolution under the heat flow of the initial data. This is essential for the method to work, because this term will appear in the exponential 
	after using the Grönwall lemma, and we must avoid exponential growth if it appears the factor $N$ (see Remark \ref{rem:proof_induction}). 

We split the first integral in $I_3$ as 
\begin{align*}
&- \int_{\R^3} \left(e^{\eta t \Delta}\partial^{\beta}u_0\cdot\nabla\right)\partial^{\alpha-\beta}v\cdot\partial^\alpha\pob{v}
\\
&= - \int_{\R^3} \left(e^{\eta t \Delta}\partial^{\beta}u_0\cdot\nabla\right)\partial^{\alpha-\beta}\pib{v}\cdot\partial^\alpha\pob{v}
- \int_{\R^3} \left(e^{\eta t \Delta}\partial^{\beta}u_0\cdot\nabla\right)\partial^{\alpha-\beta}\pob{v}\cdot\partial^\alpha\pob{v}\,
\end{align*}
and we estimate at low frequency
\begin{align*}
\left| \int_{\R^3} \left(e^{\eta t \Delta}\partial^{\beta}u_0\cdot\nabla\right)\partial^{\alpha-\beta}\pib{v}\cdot\partial^\alpha\pob{v}\,\d x \right| & \leq  \norm{\nabla\partial^\alpha \pob{v}}{L^2} \norm{e^{\eta t \Delta}\partial^\beta u_0}{L^\infty}
\| \partial^{\alpha - \beta} \pib{v}\|_{L^2}
\\ 
& \leq  \norm{\nabla\partial^\alpha \pob{v}}{L^2} \norm{e^{\eta t \Delta}\partial^\beta u_0}{L^\infty}
\| v\|_{L^2}
\\
& \leq  \frac{\epsilon}{2} \norm{\nabla\partial^\alpha \pob{v}}{L^2}^2 + \rho^2 \norm{e^{\eta t \Delta}\partial^\beta u_0}{L^\infty}^2
\end{align*}
where we have integrated by parts the nabla to justify the first inequality, we have used \eqref{BPIN2} in the second inequality
and $$\|v(t)\|^2_{L^2} \leq e_0(t)\lesssim\rho^2$$ and Cauchy-Schwartz in the last inequality.

At  high frequencies we estimate
\begin{align*}
\left| \int_{\R^3} \left(e^{\eta t \Delta}\partial^{\beta}u_0\cdot\nabla\right)\partial^{\alpha-\beta}\pob{v}\cdot\partial^\alpha\pob{v}\,\d x \right| & \leq  \norm{\nabla\partial^\alpha \pob{v}}{L^2} \norm{e^{\eta t \Delta}\partial^\beta u_0}{L^\infty}
\| \partial^{\alpha - \beta} \pob{v}\|_{L^2}
\\
&\leq \frac{\epsilon}{2} \norm{\nabla\partial^\alpha \pob{v}}{L^2}^2 + C\norm{e^{\eta t \Delta}\partial^\beta u_0}{L^\infty}^2\pob{e}_{\abs{\alpha-\beta}}(t)
\end{align*}
In conclusion 
$$
\left| \int_{\R^3} \left(e^{\eta t \Delta}\partial^{\beta}u_0\cdot\nabla\right)\partial^{\alpha-\beta}v\cdot\partial^\alpha\pob{v} \right|
\leq \epsilon\norm{\nabla\partial^\alpha \pob{v}}{L^2}^2 + C\norm{e^{\eta t \Delta}\partial^\beta u_0}{L^\infty}^2
(\rho^2 + \pob{e}_{\abs{\alpha-\beta}}(t))
$$
and we note that the induction hypothesis \eqref{eq:induction_hypothesis} gives (recall $N \geq 1$)
$$
\left| \int_{\R^3} \left(e^{\eta t \Delta}\partial^{\beta}u_0\cdot\nabla\right)\partial^{\alpha-\beta}v\cdot\partial^\alpha\pob{v} \right|
\leq \epsilon\norm{\nabla\partial^\alpha \pob{v}}{L^2}^2 + C\norm{e^{\eta t \Delta}\partial^\beta u_0}{L^\infty}^2
\rho^2N^{2|\alpha - \beta|}
$$
as long as $|\beta| \neq 0$ or $|\alpha| < r$, While in the remaining case 
$|\beta| = 0$ and $|\alpha| = r$ we have
 $$
\left| \int_{\R^3} \left(e^{\eta t \Delta} u_0\cdot\nabla\right)\partial^{\alpha}v\cdot\partial^\alpha\pob{v} \right|
\leq \epsilon\norm{\nabla\partial^\alpha \pob{v}}{L^2}^2 + C\norm{e^{\eta t \Delta} u_0}{L^\infty}^2
(\rho^2 + \pob{e}_r(t)), \qquad |\alpha| = r.
$$
The exact same argument applies, with obvious modifications, to the second, third and fourth integrals of $I_3$. We thus get,
for $|\beta| \neq 0$ or $|\alpha| < r$:
$$
\left| \int_{\R^3} \left(e^{\eta t \Delta}\partial^{\beta}u_0\cdot\nabla\right)\partial^{\alpha-\beta}h\cdot\partial^\alpha\pob{h}  \right|
\leq 
\epsilon\norm{\nabla\partial^\alpha \pob{h}}{L^2}^2 + C\norm{e^{\eta t \Delta}\partial^\beta u_0}{L^\infty}^2
\rho^2N^{2|\alpha - \beta|}
$$
$$
\left| \int_{\R^3} \left(e^{\eta t \Delta}\partial^\beta b_0\cdot\nabla\right)\partial^{\alpha-\beta}h\cdot \partial^\alpha \pob{v} \right|
\leq \epsilon\norm{\nabla\partial^\alpha \pob{v}}{L^2}^2 + C\norm{e^{\eta t \Delta}\partial^\beta b_0}{L^\infty}^2
\rho^2N^{2|\alpha - \beta|} 
$$
$$
\left| \int_{\R^3} \left(e^{\eta t \Delta}\partial^\beta b_0\cdot\nabla\right)\partial^{\alpha-\beta}v\cdot \partial^\alpha \pob{h} \right|
\leq \epsilon\norm{\nabla\partial^\alpha \pob{h}}{L^2}^2 + C\norm{e^{\eta t \Delta}\partial^\beta b_0}{L^\infty}^2
\rho^2N^{2|\alpha - \beta|} 
$$
while in the case $|\beta| = 0$ and $|\alpha| = r$:
$$
\left| \int_{\R^3} \left(e^{\eta t \Delta}u_0\cdot\nabla\right)\partial^{\alpha}h\cdot\partial^\alpha\pob{h}  \right|
\leq 
\epsilon\norm{\nabla\partial^\alpha \pob{h}}{L^2}^2 + C\norm{e^{\eta t \Delta} u_0}{L^\infty}^2
(\rho^2 + \pob{e}_r(t))
$$
$$
\left| \int_{\R^3} \left(e^{\eta t \Delta} b_0\cdot\nabla\right)\partial^{\alpha}h\cdot \partial^\alpha \pob{v} \right|
\leq \epsilon\norm{\nabla\partial^\alpha \pob{v}}{L^2}^2 + C\norm{e^{\eta t \Delta} b_0}{L^\infty}^2
(\rho^2 + \pob{e}_r(t))
$$
$$
\left| \int_{\R^3} \left(e^{\eta t \Delta} b_0\cdot\nabla\right)\partial^{\alpha}v\cdot \partial^\alpha \pob{h} \right|
\leq \epsilon\norm{\nabla\partial^\alpha \pob{h}}{L^2}^2 + C\norm{e^{\eta t \Delta} b_0}{L^\infty}^2
(\rho^2 + \pob{e}_r(t))
$$

We now move to the fifth integral of $I_3$, that we split as
 \begin{align*}
&
\int_{\R^3} \left(\partial^\beta v\cdot\nabla\right)e^{\eta t \Delta}\partial^{\alpha-\beta}u_0\cdot \partial^\alpha \pob{v}
\\
&= \int_{\R^3} \left(\partial^\beta \pib{v}\cdot\nabla\right)e^{\eta t \Delta}\partial^{\alpha-\beta}u_0\cdot \partial^\alpha \pob{v}
+
\int_{\R^3} \left(\partial^\beta \pob{v}\cdot\nabla\right)e^{\eta t \Delta}\partial^{\alpha-\beta}u_0\cdot \partial^\alpha \pob{v}\,
\end{align*}
and we estimate at low frequency
\begin{align*}
\left| \int_{\R^3} \left(\partial^\beta \pib{v}\cdot\nabla\right)e^{\eta t \Delta}\partial^{\alpha-\beta}u_0\cdot \partial^\alpha \pob{v} \right| & \leq  \norm{\nabla\partial^\alpha \pob{v}}{L^2} \norm{e^{\eta t \Delta}\partial^{\alpha - \beta} u_0}{L^\infty}
\| \partial^{\beta} \pib{v}\|_{L^2}
\\ 
& \leq  \norm{\nabla\partial^\alpha \pob{v}}{L^2} \norm{e^{\eta t \Delta}\partial^{\alpha - \beta} u_0}{L^\infty}
\| v\|_{L^2}
\\
& \leq  \frac{\epsilon}{2} \norm{\nabla\partial^\alpha \pob{v}}{L^2}^2 
+ \rho^2 \norm{e^{\eta t \Delta}\partial^{\alpha - \beta} u_0}{L^\infty}^2
\end{align*}
where we have integrated by parts the nabla to justify the first inequality, we have used \eqref{BPIN2} in the second inequality
and $$\|v(t)\|^2_{L^2} \leq e_0(t)\lesssim\rho^2$$ and Cauchy-Schwartz in the last inequality.

At  high frequencies we estimate
\begin{align*}
\left| \int_{\R^3} \left(\partial^\beta \pob{v}\cdot\nabla\right)e^{\eta t \Delta}\partial^{\alpha-\beta}u_0\cdot \partial^\alpha \pob{v} \right| & \leq  \norm{\nabla\partial^\alpha \pob{v}}{L^2} \norm{e^{\eta t \Delta}\partial^{\alpha - \beta} u_0}{L^\infty}
\| \partial^{\beta} \pob{v}\|_{L^2}
\\
&\leq \frac{\epsilon}{2} \norm{\nabla\partial^\alpha \pob{v}}{L^2}^2 
+ C\norm{e^{\eta t \Delta}\partial^{\alpha -\beta} u_0}{L^\infty}^2 \pob{e}_{\abs{\beta}}(t)
\end{align*}
In conclusion 
$$
\left| 
\int_{\R^3} \left(\partial^\beta \pob{v}\cdot\nabla\right)e^{\eta t \Delta}\partial^{\alpha-\beta}u_0\cdot \partial^\alpha \pob{v} \right|
\leq \epsilon\norm{\nabla\partial^\alpha \pob{v}}{L^2}^2 + C\norm{e^{\eta t \Delta}\partial^{\alpha - \beta} u_0}{L^\infty}^2
(\rho^2 + \pob{e}_{\abs{\beta}}(t))
$$
and we note that the induction hypothesis \eqref{eq:induction_hypothesis} gives (recall $N \geq 1$)
$$
\left| \int_{\R^3} \left(\partial^\beta \pob{v}\cdot\nabla\right)e^{\eta t \Delta}\partial^{\alpha-\beta}u_0\cdot \partial^\alpha \pob{v} \right|
\leq \epsilon\norm{\nabla\partial^\alpha \pob{v}}{L^2}^2 + C\norm{e^{\eta t \Delta}\partial^{\alpha - \beta} u_0}{L^\infty}^2
\rho^2N^{2|\beta|}
$$
as long as $\beta \neq \alpha$ or $|\alpha| < r$, while in the remaining case 
$\beta = \alpha$ and $|\alpha| = r$ we have
 $$
\left| \int_{\R^3} \left(\partial^\alpha \pob{v}\cdot\nabla\right)e^{\eta t \Delta}u_0\cdot \partial^\alpha \pob{v} \right|
\leq \epsilon\norm{\nabla\partial^\alpha \pob{v}}{L^2}^2 + C\norm{e^{\eta t \Delta} u_0}{L^\infty}^2
(\rho^2 + \pob{e}_r(t)), \qquad |\alpha| = r.
$$ 

The exact same argument applies, with obvious modifications, to the sixth, seventh, eighth integrals of $I_3$. We thus get,
for $\beta \neq \alpha$ or $|\alpha| < r$:
$$
\left| \int_{\R^3} \left(\partial^\beta h\cdot\nabla\right)e^{\eta t \Delta}\partial^{\alpha-\beta}u_0\cdot \partial^\alpha \pob{h} \right|
\leq \epsilon\norm{\nabla\partial^\alpha \pob{h}}{L^2}^2 + C\norm{e^{\eta t \Delta}\partial^{\alpha - \beta} u_0}{L^\infty}^2
\rho^2N^{2|\beta|}
$$
$$
\left| \int_{\R^3}\left(\partial^\beta h\cdot\nabla\right)e^{\eta t \Delta}\partial^{\alpha-\beta}b_0\cdot \partial^\alpha \pob{v} \right|
\leq \epsilon\norm{\nabla\partial^\alpha \pob{v}}{L^2}^2 + C\norm{e^{\eta t \Delta}\partial^{\alpha - \beta} b_0}{L^\infty}^2
\rho^2N^{2|\beta|}
$$
$$
\left| \int_{\R^3} \left(\partial^\beta v\cdot\nabla\right)e^{\eta t \Delta}\partial^{\alpha-\beta}b_0\cdot \partial^\alpha \pob{h}\right|
\leq \epsilon\norm{\nabla\partial^\alpha \pob{h}}{L^2}^2 + C\norm{e^{\eta t \Delta}\partial^{\alpha - \beta} b_0}{L^\infty}^2
\rho^2N^{2|\beta|},
$$
while for
 $\beta = \alpha$ and $|\alpha| = r$:
$$
\left| \int_{\R^3} \left(\partial^\alpha h\cdot\nabla\right)e^{\eta t \Delta}u_0\cdot \partial^\alpha \pob{h} \right|
\leq \epsilon\norm{\nabla\partial^\alpha \pob{h}}{L^2}^2 +  C\norm{e^{\eta t \Delta} u_0}{L^\infty}^2
(\rho^2 + \pob{e}_r(t))
$$
$$
\left| \int_{\R^3}\left(\partial^\alpha h\cdot\nabla\right)e^{\eta t \Delta}b_0\cdot \partial^\alpha \pob{v} \right|
\leq \epsilon\norm{\nabla\partial^\alpha \pob{v}}{L^2}^2 +  C\norm{e^{\eta t \Delta} b_0}{L^\infty}^2
(\rho^2 + \pob{e}_r(t))
$$
$$
\left| \int_{\R^3} \left(\partial^\alpha v\cdot\nabla\right)e^{\eta t \Delta}b_0\cdot \partial^\alpha \pob{h}\right|
\leq \epsilon\norm{\nabla\partial^\alpha \pob{h}}{L^2}^2 +  C\norm{e^{\eta t \Delta} b_0}{L^\infty}^2
(\rho^2 + \pob{e}_r(t))
$$
Putting all the estimates together, summing over $\abs{\alpha}\leq r$ we arrive to
\begin{equation*}
		\begin{split}
			& \frac{\d}{\d t}\pob{e}_r(t) + 2 \eta \sum_{\abs{\alpha}\leq r} \int_{\R^3} \Bigg(|\nabla\partial^\alpha \pob{v}|^2 + \abs{\nabla\partial^\alpha \pob{h}}^2\Bigg) \\
			&\leq 
			C_1 \varepsilon  \sum_{\abs{\alpha}\leq r} \int_{\R^3} \Bigg(|\nabla\partial^\alpha \pob{v}|^2 + \abs{\nabla\partial^\alpha \pob{h}}^2\Bigg) 
			\\ &
			+ C_2 \left(
			 \rho^4 + \rho^2\pob{e}_r(t) + \pob{e}_r(t)^3 + \left(\norm{e^{\eta t \Delta}u_0}{L^\infty}^2 
			+ \norm{e^{\eta t \Delta}b_0}{L^\infty}^2\right)\pob{e}_r(t) \right) \\
			&\quad+ C_2 \rho^2 \sum_{m=0}^rN^{2(r-m)}\left(\norm{e^{\eta t \Delta}\nabla^m u_0}{L^\infty}^2 + \norm{e^{\eta t \Delta}\nabla^m b_0}{L^\infty}^2\right), \\
		\end{split}
	\end{equation*}
where the constant $C_1$ is $r$-dependent, while the constant $C_2$ depends on $r$ and $\varepsilon >0$. 
Choosing $$\varepsilon = \frac{\eta}{C_1}$$ the estimate becomes 	
\begin{equation}\label{eq:sum.alpha.all.estim.e_r}
		\begin{split}
			\frac{\d}{\d t}\pob{e}_r(t) & \leq - \eta \sum_{\abs{\alpha}\leq r} \int_{\R^3} \Bigg(|\nabla\partial^\alpha \pob{v}|^2 + \abs{\nabla\partial^\alpha \pob{h}}^2\Bigg) \\
			&+ C_2 \rho^4 + \rho^2\pob{e}_r(t) + C_2 \pob{e}_r(t)^3 + C_2 \left(\norm{e^{\eta t \Delta}u_0}{L^\infty}^2 
			+  \norm{e^{\eta t \Delta}b_0}{L^\infty}^2\right)\pob{e}_r(t) \\
			&+ C_2 \rho^2 \sum_{m=0}^rN^{2(r-m)}\left(\norm{e^{\eta t \Delta}\nabla^m u_0}{L^\infty}^2 
			+ \norm{e^{\eta t \Delta}\nabla^m b_0}{L^\infty}^2\right). \\
		\end{split}
	\end{equation}
	The advantage of the high frequencies regime on $\R^3$ is that at this point we can exploit Poincare's inequality 
	(see \eqref{BPIN}):
	$$\sum_{\abs{\alpha}\leq r} \int_{\R^3} \left(\abs{\nabla\partial^\alpha \pob{v}}^2 + \abs{\nabla\partial^\alpha \pob{h}}^2\right)\d x \geq \frac12 \sum_{\abs{\alpha}\leq r} \int_{\R^3} \left(\abs{\partial^\alpha \pob{v}}^2 + \abs{\partial^\alpha \pob{h}}^2\right)\d x = \frac12 \pob{e}_r(t)$$
	to deduce
	$$- \eta \sum_{\abs{\alpha}\leq r} \int_{\R^3} \left(\abs{\nabla\partial^\alpha \pob{v}}^2 + \abs{\nabla\partial^\alpha \pob{h}}^2\right) \leq - \frac12 \eta \pob{e}_r(t).$$
	Plugging this bound into \eqref{eq:sum.alpha.all.estim.e_r} we arrive to
	\begin{equation}\label{eq:sum.alpha.all.estim.e_r2}
		\begin{split}
			\frac{\d}{\d t}\pob{e}_r(t) & \leq - \frac12 \eta \pob{e}_r(t) \\
			&+ C_2 \rho^4 + C_2\rho^2\pob{e}_r(t) + C_2 \pob{e}_r(t)^3 + C_2 \left(\norm{e^{\eta t \Delta}u_0}{L^\infty}^2 
			+  \norm{e^{\eta t \Delta}b_0}{L^\infty}^2\right)\pob{e}_r(t) \\
			&+ C_2 \rho^2 \sum_{m=0}^rN^{2(r-m)}\left(\norm{e^{\eta t \Delta}\nabla^m u_0}{L^\infty}^2 
			+ C_2 \norm{e^{\eta t \Delta}\nabla^m b_0}{L^\infty}^2\right). \\
		\end{split}
	\end{equation}
	
	Let $c >0$ a (small) quantity that will be fixed later. As $\pob{e}_r(0)=0$, by continuity in time of 
	the (strong) solution it make sense to define the maximal time 
	$T$ such that one has $\pob{e}_r(t) \leq c$ for all $t \in [0, T]$. If $T \neq + \infty$, then a continuity argument shows
	$\pob{e}_r(T) = c$ (we will in fact prove that $T= + \infty$). 
	Recalling that $\rho \leq cN^{-r-1}$ and $N \geq 1$, we see that 
 we can choose  $c>0$ sufficiently small in such a way that, for all $t \in [0, T]$ we have:
	\begin{equation}\label{eq:assumption_sigma_er}
		- \frac12 \eta \pob{e}_r(t)+C_2 \pob{e}_r(t)^3+C_2 \rho^2\pob{e}_r(t) \leq - \frac14 \eta \pob{e}_r(t).
	\end{equation}
	Plugging this inequality into \eqref{eq:sum.alpha.all.estim.e_r2} we see that for all $t \in [0, T]$: 
	\begin{equation}
		\begin{split}
			\frac{\d}{\d t}\pob{e}_r(t) & \leq - \frac14 \eta \pob{e}_r(t) 
			   + C_2 \left(\norm{e^{\eta t \Delta}u_0}{L^\infty}^2 
			+  \norm{e^{\eta t \Delta}b_0}{L^\infty}^2\right)\pob{e}_r(t) \\
			&+ C_2 \rho^4 + C_2 \rho^2 \sum_{m=0}^rN^{2(r-m)}\left(\norm{e^{\eta t \Delta}\nabla^m u_0}{L^\infty}^2 
			+  \norm{e^{\eta t \Delta}\nabla^m b_0}{L^\infty}^2\right). \\
		\end{split}
	\end{equation}
	
	Applying Grönwall's inequality, we obtain for all $t \in [0, T]$:
	\begin{equation}\label{eq:last_ineq_pob_er}
		\pob{e}_r(t) \lesssim  (\rho^4+\rho^2N^{2r}) \int_0^t e^{\int_s^t \left(-\frac14 \eta + \left(\norm{e^{\tau \eta \Delta}u_0}{L^\infty}^2 + \norm{e^{\tau \eta \Delta}b_0}{L^\infty}^2\right)\right)\d\tau}\d s \lesssim \rho^2N^{2r}.
	\end{equation}
	where the hypothesis \eqref{eq:hypothesis_data_for_global_existence} has been used to bound the 
	$L^2W^{m,\infty}$ norm of the linear propagators that came from the induction terms, and its $L^2L^\infty$ norm in the exponential. 
	In conclusion we have proved that  for all $t \in [0, T]$:
	$$
		\pob{e}_r(t) \leq C \rho^2N^{2r}.
	$$
	Since $N \geq 1$ and $\rho \leq cN^{-r-1}$, we see that, choosing  $c$ even smaller (if needed), we have indeed
	 $C \rho^2 N^{2r} \leq Cc^2N^{-2} \leq \frac12 c.$
	 Thus $\pob{e}_r(T) \leq \frac12 c$, that is in contradiction with~$\pob{e}_r(T) = c$, that we have observed to be the case if 
	 $T < +\infty$. This proves that $T = + \infty$ and that inequality \eqref{eq:last_ineq_pob_er} can be thus extended 
	 to all times.

Then we notice that, by inequality \eqref{BPIN2} we have
	$$\pib{e}_r(t)\leq\pib{e}_0(t)\leq C\rho^2.$$
	Thus
	$$e_r(t)=\pib{e}_r(t)+\pob{e}_r(t)\leq C\rho^2N^{2r}, \quad\forall t\geq0,$$
	that is indeed the desired inequality \eqref{eq:induction_hypothesis}.
	This concludes the proof.
\end{proof}

\begin{remark}\label{rem:proof_induction}
	Three facts may be pointed out:
	\begin{enumerate}
		\item In all the estimates that, like \eqref{eq:Gronwall_exp_r0}, will give an  exponential part when applying Grönwall's inequality, and it is crucial for the magnetic reconnection result that no derivatives, and thus no powers of $N$ appear in the exponent, because this exponential growth cannot be killed by the parameter $\rho$ at the end of the argument (see the choice of the parameters at the end of the proof of Theorem \ref{thm:large_magnetic_reconnection}).
		
		\item The $r=0$ estimate does not deppend on $N$, which is also important when carrying out the Poincaré argument in order to kill the high powers of the energy with the linear one in \eqref{eq:assumption_sigma_er}.
		
		\item The projection inside of the ball is comparable with the $L^2$-case always, and this is needed for the last step of the proof. If not, we should have carried another energy-type argument for the derivative of $\pib{e}_r(t)$ and in this case we would not have Poincaré's inequality to obtain decay from Grönwall.
	\end{enumerate}
\end{remark}

\section{Magnetic reconnection for large data}
\noindent
The other main result of this article provides an analytic example of a class of large initial data in $\Dot{B}_{\infty,\infty}^{-1}$ (the largest critical space for \eqref{eq:MHD}) such that magnetic reconnection happens. This is something that has never been done before; in other works concerning this problem (see \cite{CL24}) analytic examples for reconnection have been built but for small initial data, to guarantee the existence of the solution for all times. In contrast, here we are going to impose the smallness condition not on the initial data $(u_0,b_0)$ itself but on the difference $u_0-b_0$, and exploit the symmetries of the system of equations. As in \cite{CL24}, we use as a topological constraint the number of hyperbolic critical points of the field at different times. One can consult \cite{PP22,PT00} to get some insight into magnetic reconnection and the topological changes of the field lines at null points during this process.
\begin{remark}
	For this result it is interesting to keep track of the resistivity to understand better the timescale for the reconnection. However, as pointed out in Remark \ref{rem:resistivity_equal_viscosity}, in the proof of this result it is crucial to use some cancellations of large terms that appear thanks to the symmetries of the system. In order to maintain them, we must set the viscosity equal to the resistivity, $\nu=\eta$. We believe that if $\nu=\eta+\epsilon$ for $0<\epsilon\ll1$, a similar method should work to prove both Theorem \ref{thm:global_existence_large_data} and Theorem \ref{thm:large_magnetic_reconnection}. It is left as an open problem to extend this result to the case where $\nu\neq\eta$.
\end{remark}

\subsection{Construction of the initial data}
Define the initial data
\begin{equation}\label{eq:def_large_initial_data}
	u_0:=M\text{curl}(\phi B_N) \quad \text{and} \quad b_0:=M\curl{\phi B_N}+\rho e^{-\eta T\Delta}\curl{\psi W}
\end{equation}
making the construction with the following in mind: our aim is to localize the Beltrami field in order to obtain finite energy solutions, preserving the incompressibility condition by putting the curl in front of the localization. To choose the parameters, define
\begin{equation}\label{eq:def_b01}
	\phi(x):=(1+\abs{x}^2)^{-\alpha} \quad\text{and}\quad B_N(x):=(\sin (Nx_3), \cos (Nx_3), 0),
\end{equation}
for some $N\gg\alpha\geq3/2$, while
\begin{equation}\label{eq:def_b02}
	\psi(x):=e^{-\frac{\abs{x}^2}{8\eta T}} \quad\text{and}\quad W(x):=(\cos x_2,\cos x_3, \cos x_1).
\end{equation}
With these definitions, the global existence of a (unique) strong solution $(u,b)$ to \eqref{eq:MHD} starting at $(u_0,b_0)\in H^r(\R^3)$ is ensured by means of Theorem \ref{thm:global_existence_large_data}. To check this, we need some technical lemmas:
\begin{lem}\label{lem:FT.params}
    Let $\delta$ denote the delta distribution centered at zero.
    \begin{itemize}
        \item The Fourier transforms of $B_N$ is given by
        \begin{equation}\label{eq:FT.B_N}
            \widehat{B_N}(\xi)=\sum_{l=1}^4\widehat{B_N^l}(\xi),
        \end{equation}
        where
    	\begin{equation*}
    		\begin{aligned}
    			\widehat{B_N^1}(\xi):=\frac{1}{2\i}(\delta(\xi_1)\otimes\delta(\xi_2)\otimes\delta(\xi_3-N),0,0), & \quad \widehat{B_N^2}(\xi):=-\frac{1}{2\i}(\delta(\xi_1)\otimes\delta(\xi_2)\otimes\delta(\xi_3+N),0,0), \\
    			\widehat{B_N^3}(\xi):=\frac{1}{2}(0,\delta(\xi_1)\otimes\delta(\xi_2)\otimes\delta(\xi_3-N),0), &\quad \widehat{B_N^4}(\xi):=\frac{1}{2}(0,\delta(\xi_1)\otimes\delta(\xi_2)\otimes\delta(\xi_3+N),0),
    		\end{aligned}
    	\end{equation*}
this follows once one express the sine and cosine in terms of exponential and using the identity 
$\widehat{\delta(\cdot + a)}(\xi) = e^{-i\xi a}$; see for instance the computation for $(B_N)_1$ below.
        \item The Fourier transform of $W$ is given by
        \begin{equation*}
            2\widehat{W}(\xi) = \sum_{j=1}^3 \widehat{W_j}(\xi)\,,
        \end{equation*}
        where
        \begin{equation*}
            \begin{split}
                \widehat{W_1}(\xi) &= \delta(\xi_1) \otimes \big( \delta(\xi_2-1) + \delta(\xi_2+1) \big) \otimes \delta(\xi_3) \,, \\
                \widehat{W_2}(\xi) &= \delta(\xi_1) \otimes \delta(\xi_2) \otimes \big( \delta(\xi_3-1) + \delta(\xi_3+1) \big) \,, \\
                \widehat{W_3}(\xi) &= \big( \delta(\xi_1-1) + \delta(\xi_1+1) \big) \otimes \delta(\xi_2) \otimes \delta(\xi_3) \,.
            \end{split}
        \end{equation*}

        \item The Fourier transform of $\psi$ is given by
        \begin{equation*}
            \widehat{\psi}(\xi) = 8(2\eta T)^{3/2}e^{-2\eta T |\xi|^2}\,.
        \end{equation*}
    \end{itemize}
\end{lem}
\begin{proof}
    For $B_N$, we state the proof given in \cite[Proposition 4.3]{CL24} for the sake of completeness. We rewrite the first component as
    \begin{equation*}
        (B_N)_1 = \sin(Nx_3) = \frac{e^{\i Nx_3} - e^{\i Nx_3}}{2\i}.
    \end{equation*}
    Then, compute
    \begin{equation*}
        \begin{split}
            \mathcal{F}[e^{\i Nx_3}](\xi) &= \int_{\R^3} e^{-\i x\cdot\xi} e^{\i Nx_3}\,\d x \\
            &= \bigg( \int_{\R} e^{-\i x_1\xi_1}\,\d x_1 \bigg) \bigg( \int_{\R} e^{-\i x_2\xi_2}\,\d x_2 \bigg) \bigg( \int_{\R} e^{-\i x_3(\xi_3-N)}\,\d x_1 \bigg) \\
            &= \delta(\xi_1) \otimes \delta(\xi_2) \otimes \delta(\xi_3-N).
        \end{split}
    \end{equation*}
    where we have used Fubini's theorem in the second equality and the fact that $\hat{1}=\delta$ in the third equality. The expression \eqref{eq:FT.B_N} follows by applying the same calculus to $e^{-\i Nx_3}$ and to $(B_N)_2 = \cos(Nx_3)$, and a similar computation works for $W$ (writing $1$ instead of $N$).

    \medskip

    Finally, for $\psi$, we compute
    \begin{equation*}
        \begin{split}
            \widehat{\psi}(\xi) = \int_{\R^3} e^{-\i x\cdot\xi}e^{-|x|^2/(8\eta T)}\,\d x &= \prod_{k=1}^3 \int_{\R} e^{-\i x_k\xi_k} e^{-x_k^2 / (8\eta T)}\,\d x_k \\
            &= \prod_{k=1}^3 \int_{\R} e^{-\frac{1}{8\eta T} \big( x_k + \i 4\eta T \xi_k \big)^2} e^{- 2\eta T \xi_k^2}\,\d x_k \\
            &= \prod_{k=1}^3 e^{- 2\eta T \xi_k^2} \int_{\R} e^{-\frac{y_k^2}{8\eta T}}\,\d y_k \quad \bigg( y_k := x_k+\i 4\eta T\xi_k \bigg) \\
            &= (8\pi\eta T)^{3/2} e^{-4\eta T |\xi|^2}.
        \end{split}
    \end{equation*}
\end{proof}

\begin{lem}\label{lem:norm_estimates_low_times}
	Let $t\geq0$. Then, for any $n\in\N$, $1\leq p\leq\infty$ and $r\in\N$,
	\begin{align}
		\norm{e^{\eta t\Delta}\curl{\phi B_N}}{L^p} &\lesssim_n e^{-c\eta tN^2}N + N^{-n} \label{eq:estim_norm_evolution_phiBN_Lp_small_times}, \\
		\norm{e^{\eta (t-T)\Delta}\curl{\psi W}}{H^r} &\leq C(T,\eta). \label{eq:estim_norm_evolution_psiW_Lp_small_times}
	\end{align}
\end{lem}
\begin{proof}
	Estimate \eqref{eq:estim_norm_evolution_phiBN_Lp_small_times} is proved by using Bernstein inequalities. Observe that
	$$\phi=P_{\leq N/2}\phi + P_{>N/2}\phi,$$
	meaning that we project (in the Fourier side) inside and outside of the ball $B(0,N/2)$, respectively (following the notations in Proposition \ref{prop:Bernstein_inequalities}). For the part inside of the ball, as $\widehat{B_N^1}$ is supported in the point $(0,0,N)$ ((by Lemma \ref{lem:FT.params}; the same for $\widehat{B_N^3}$ and in $(0,0,-N)$ for $\widehat{B_N^2}$ and $\widehat{B_N^4}$), then
	$$\text{supp}\mathcal{F}[\Phi_N] \subseteq (0,0,N)+B(0,N/2) \subseteq N\mathcal{A}, \quad \Phi_N:=(P_{\leq N/2}\phi)B_N^1 \;\text{ and }\; \mathcal{A}:=\left\{\frac{1}{2}\leq\abs{\xi}\leq\frac{3}{2}\right\}.$$
	Thus, by Propositions \ref{prop:Bernstein_inequalities} and \ref{prop:heat.act.annulus}, together with Hölder's inequality and the facts that $\curl{g B_N}=\nabla g\wedge B_N + g\curl{B_N}=\nabla g\wedge B_N + NgB_N$ (for $g=P_{\leq N/2}\phi$) and $\norm{B_N}{L^{\infty}}=1$, we obtain that
    \begin{equation*}
        \begin{split}
            \norm{e^{\eta t\Delta}\curl{\Phi_N}}{L^p} &\lesssim e^{-c\eta t N^2}\norm{\Phi_N}{L^p} \\
            &\lesssim e^{-c\eta t N^2} \big( \norm{\nabla(P_{\leq N/2}\phi) \wedge B_N^1}{L^p} + N\norm{(P_{\leq N/2}\phi) B_N^1}{L^p} \big) \\
            &\lesssim e^{-c\eta t N^2} \big( \norm{\nabla(P_{\leq N/2}\phi)}{L^p} \norm{B_N^1}{L^{\infty}} + N\norm{P_{\leq N/2}\phi}{L^p} \norm{B_N^1}{L^{\infty}} \big) \\
            &\lesssim e^{-c\eta t N^2} 2N\norm{P_{\leq N/2}\phi}{L^p} \\
            &\lesssim Ne^{-c\eta t N^2},
        \end{split}
    \end{equation*}
	where the last norm is finite by the smoothness of $\phi$. For the part outside of the ball, denote $\varphi_N:=(P_{>N/2}\phi)B_N^1$. By the continuity of the heat kernel,
	$$\norm{e^{\eta t\Delta}\curl{\varphi_N}}{L^p} \lesssim \norm{\curl{\varphi_N}}{L^p}.$$
	Moreover, denote by $f:=P_{>N/2}\phi$ and observe that
	$$
	\norm{f}{L^p}\lesssim_k N^{-k}\norm{P_{>N/2}\abs{D}^k\phi}{L^p} \lesssim_k N^{-k}
	$$
	by Proposition \ref{prop:Bernstein_inequalities}. Choose $k=n+1$ and use that $\curl{f B_N}=\nabla f\wedge B_N + f\curl{B_N}=\nabla f\wedge B_N + NfB_N$ to obtain that
	$$\norm{\curl{\varphi_N}}{L^p} \leq \norm{\nabla f}{L^p}\norm{B_N^1}{L^\infty} + N\norm{f}{L^p}\norm{B_N^1}{L^\infty} \lesssim_{n} N^{-n}.$$
	We end up with
	$$\norm{e^{\eta t\Delta}\curl{\phi B_N^1}}{L^p} \lesssim_n e^{-c\eta tN^2}N + N^{-n}, \quad \forall t\geq0,\; \forall n\in\N,\; \forall1\leq p\leq\infty.$$
	For the estimate \eqref{eq:estim_norm_evolution_psiW_Lp_small_times} we compute explicitly the $H^r$-norm using its (Fourier) definition and Lemma \ref{lem:FT.params}:
	\begin{align}\label{fdsjklgjngsnjfk}
		\norm{e^{\eta(t-T)\Delta}\curl{\psi W}}{H^r}^2 &\simeq \int_{\R^3} (1+\abs{\xi}^2)^r e^{2\eta(T-t)\abs{\xi}^2} \abs{\xi}^2 \abs{\widehat{\psi W}(\xi)}^2\,\d\xi \\ \nonumber
		&\lesssim \int_{\R^3} (1+\abs{\xi}^2)^{r+1} e^{2\eta(T-t)\abs{\xi}^2} e^{-4\eta T(\abs{\xi_1}^2+\abs{\xi_2-1}^2+\abs{\xi_3}^3)}\,\d\xi \\ \nonumber
		&\lesssim \int_{\R^3} (1+\abs{\xi}^2)^{r+1} e^{2\eta T\abs{\xi}^2} e^{-4\eta T(\abs{\xi_1}^2+\abs{\xi_2-1}^2+\abs{\xi_3}^3)}\,\d\xi \\ \nonumber
		&\lesssim \int_{\R^3} (1+\abs{\xi}^2)^{r+1} e^{-2\eta T\abs{\xi_1}^2}e^{-2\eta T\abs{\xi_2-2}^2} e^{-2\eta T\abs{\xi_3}^2}\,\d\xi
	\end{align}
	and this integral is bounded by an absolute constant (that is inversely proportional to both parameters $T$ and $\eta$).
\end{proof}
When the time is large enough, we can improve these estimates:
\begin{lem}\label{lem:norm_estimates_large_times}
	Let $t>2T$. Then, for any $1\leq p\leq\infty$ and $r\in\N$,
	\begin{align}
		\norm{e^{\eta t\Delta}\curl{\phi B_N}}{L^p} &\lesssim e^{-c\eta tN^2}N + t^{-\frac{1}{2}-\frac{3}{2}\left(1-\frac{1}{p}\right)}, \label{eq:estim_norm_evolution_phiBN_Lp_large_times}\\
		\norm{e^{\eta(t-T)\Delta}\curl{\psi W}}{L^p} &\lesssim (t-T)^{-\frac{1}{2}-\frac{3}{2}\left(1-\frac{1}{p}\right)}, \label{eq:estim_norm_evolution_psiW_Lp_large_times} \\
		\norm{e^{\eta(t-T)\Delta}\curl{\psi W}}{H^r} &\lesssim t^{-1/4}+t^{-r-5/4}. \label{eq:estim_norm_evolution_psiW_Hr_large_times}
	\end{align}
\end{lem}
\begin{proof}
	For estimate \eqref{eq:estim_norm_evolution_phiBN_Lp_large_times}, the low Fourier modes can be handled as in Lemma \ref{lem:norm_estimates_low_times}, while for the high Fourier modes we can just use the heat kernel estimates:
	$$\norm{e^{\eta t\Delta}\curl{\phi B_N}}{L^p} \lesssim t^{-\frac{1}{2}-\frac{3}{2}\left(1-\frac{1}{p}\right)}\norm{\phi B_N}{L^1} \lesssim t^{-\frac{1}{2}-\frac{3}{2}\left(1-\frac{1}{p}\right)}.$$
	Since we are in the regime where $t-T>0$, by Proposition \ref{prop:heat.dispersion}, we have that
	$$\norm{e^{\eta(t-T)\Delta}\curl{\psi W}}{L^p} \lesssim (t-T)^{-\frac{1}{2}-\frac{3}{2}\left(1-\frac{1}{p}\right)}\norm{\psi W}{L^1} \lesssim (t-T)^{-\frac{1}{2}-\frac{3}{2}\left(1-\frac{1}{p}\right)}.$$
	Finally, for the $H^r$-norm, as we are imposing that $t>2T$, we have that $2(T-t)<-t$ and so
	\begin{align*}
		\norm{e^{\eta(t-T)\Delta}\curl{\psi W}}{H^r}^2 &\simeq \int_{\R^3} (1+\abs{\xi}^2)^r e^{2\eta(T-t)\abs{\xi}^2} \abs{\xi}^2 \abs{\widehat{\psi W}(\xi)}\,\d\xi \\
		&\lesssim \int_{\R^3} (1+\abs{\xi}^2)^{r+1} e^{2\eta(T-t)\abs{\xi}^2} e^{-4\eta T(\abs{\xi_1}^2+\abs{\xi_2-1}^2+\abs{\xi_3}^2)}\d\xi \\
		&\lesssim \int_{\R^3} (1+\abs{\xi}^2)^{r+1} e^{-\eta t\abs{\xi}^2} \d\xi \\
		&\lesssim \int_0^\infty (1+\zeta^2)^{r+2}e^{-\eta t\zeta^2}\d\zeta \\
		&\lesssim t^{-1/2}+t^{-r-5/2}.
	\end{align*}
	We have calculated this last integral by induction: our claim is that
	$$\int_0^\infty (1+y^2)^ke^{-ay^2}\,\d y \lesssim a^{-1/2}+a^{-k-1/2}.$$
    To prove it, we first use that $(1+y^2)^{k} \leq 2^{k-1}(1+y^{2k})$ to separate the integrals as
    \begin{equation*}
        \int_0^\infty (1+y^2)^k e^{-ay^2}\,\d y \lesssim \int_0^\infty e^{-ay^2}\,\d y + \int_0^\infty y^{2k} e^{-ay^2}\,\d y.
    \end{equation*}
    Now, it suffices to prove that 
    \begin{equation}\label{eq:ind.hip.int.exp}
        \forall k\in\N_0, \quad \int_0^\infty y^{2k} e^{-ay^2}\,\d y \lesssim a^{-k-1/2}.
    \end{equation}
 We will do it by induction over $k$, the case $k=0$ being well-known. Assume that \eqref{eq:ind.hip.int.exp} is true for $k=m$ and let us prove that it holds for $k=m+1$: using integration by parts,
    \begin{multline*}
        \int_0^\infty y^{2(m+1)} e^{-ay^2}\,\d y = -\frac{1}{2a} \int_0^\infty y^{2m+1}\frac{\d}{\d y}e^{-ay^2}\,\d y \\
        = \frac{2m+1}{2a} \int_0^\infty y^{2m} e^{-ay^2}\,\d y \overset{\text{h.i.}}{\lesssim} \frac{1}{2a}a^{-m-1/2} \lesssim a^{-(m+1)-1/2}.
    \end{multline*}
\end{proof}
If one wants to introduce the derivatives, it can be easily checked with the same proof of the previous lemmas that
\begin{align}
	\norm{e^{\eta t\Delta}\partial^\alpha\curl{\phi B_N}}{L^p} &\lesssim e^{-c\eta tN^2}N^{\abs{\alpha}+1} + N^{-n}, \label{eq:estim_norm_evolution_alpha-phiBN_Lp_small_times}\\
	\norm{e^{\eta(t-T)\Delta}\partial^\alpha\curl{\psi W}}{H^r} &\lesssim C(T,\eta,\abs{\alpha}) \label{eq:estim_norm_evolution_alpha-psiW_Lp_small_times}
\end{align}
for all times $t\geq0$ and
\begin{align}
	\norm{e^{\eta t\Delta}\partial^\alpha\curl{\phi B_N}}{L^p} &\lesssim e^{-c\eta tN^2}N^{\abs{\alpha}+1} + t^{-\frac{\abs{\alpha}+1}{2}-\frac{3}{2}\left(1-\frac{1}{p}\right)}, \label{eq:estim_norm_evolution_alpha-phiBN_Lp_large_times}\\
	\norm{e^{\eta (t-T)\Delta}\partial^\alpha\curl{\psi W}}{L^p} &\lesssim (t-T)^{-\frac{\abs{\alpha}+1}{2}-\frac{3}{2}\left(1-\frac{1}{p}\right)} \label{eq:estim_norm_evolution_alpha-psiW_Lp_large_times}
\end{align}
when $t>2T$.
\begin{remark}
	If one keeps track of the constants, in the end they are inversely proportional to the resistivity.
\end{remark}
\begin{prop}
	The initial data $(u_0,b_0)$ defined in \eqref{eq:def_large_initial_data} satisfies the hypothesis \eqref{eq:hypothesis_data_for_global_existence} of the Theorem \ref{thm:global_existence_large_data}.
\end{prop}
\begin{proof}
	On the one hand, the smallness on the difference is trivially satisfied:
	$$\norm{u_0-b_0}{H^r} = \rho\norm{e^{-\eta T\Delta}\curl{\psi W}}{H^r} \lesssim \rho.$$
	On the other hand,
	\begin{equation*}
		\begin{split}
			\norm{e^{\eta t\Delta}u_0}{W^{r,\infty}}^2 &\simeq \sum_{m=0}^r M\norm{e^{\eta t\Delta}\nabla^m\curl{\phi B_N}}{L^\infty}^2 \\
			&\lesssim e^{-2c\eta tN^2}N^{2(r+1)}+N^{-2n}\chi_{[0,2T]}(t) + t^{-(r+1)-3}\chi_{(2T,\infty)}(t),
		\end{split}
	\end{equation*}
	where we have used the estimates \eqref{eq:estim_norm_evolution_alpha-phiBN_Lp_small_times} and \eqref{eq:estim_norm_evolution_alpha-phiBN_Lp_large_times} with $p=\infty$. Then,
	$$\norm{e^{\eta t\Delta}u_0}{L_t^2W_x^{r,\infty}}^2 \lesssim \int_0^\infty \left(e^{-2c\eta sN^2}N^{2(m+1)}+N^{-2n}\chi_{[0,2T]} + s^{-(m+1)-3}\chi_{(2T,\infty)}\right)\,\d s \lesssim 1+N^{2m}.$$
	The same works for $\norm{e^{\eta t\Delta}b_0}{L_t^2W_x^{r,\infty}}$ because the part involving $\rho\norm{e^{-\eta (T-t)\Delta}\nabla^m\curl{\psi W}}{L^\infty}$ can be controlled by a factor of $\rho$ using the estimates \eqref{eq:estim_norm_evolution_alpha-psiW_Lp_small_times} and \eqref{eq:estim_norm_evolution_alpha-psiW_Lp_large_times}.
\end{proof}

Moreover, the initial data satisfies the following properties:
\begin{lem}\label{lem:curl.phi.BN.crit.points}
	Let $N$ be sufficiently large with respect to $\alpha$. Then, the field $\curl{\phi B_N}$ has no critical points.
\end{lem}
\begin{proof}
	Observe that $\curl{\phi B_N} = N\phi B_N+(\nabla\phi)\wedge B_N$ and that
	\begin{multline*}
		\abs{N\phi(x)B_N(x)+(\nabla\phi)\wedge B_N(x)} = \abs{N\phi(x)B_N(x)-\frac{2\alpha x\wedge B_N(x)}{(1+\abs{x}^2)^{\alpha+1}}} \\
		\geq \frac{N}{(1+\abs{x}^2)^\alpha}-\frac{2\alpha}{(1+\abs{x}^2)^{\alpha+1/2}} \gg c\frac{N}{(1+\abs{x}^2)^\alpha}
	\end{multline*}
	for some constant $0<c<1$, where we have used the inverse triangular inequality, that $\abs{B_N}=1$ and the fact that $N\gg\alpha$. In particular, this quantity never vanishes.
\end{proof}
\begin{lem}
	The field $\curl{\psi W}$ has $(0,0,0)$ as a hyperbolic point.
\end{lem}
\begin{proof}
	Compute $\curl{\psi W} = \psi\,\curl{W} + (\nabla \psi)\wedge W$. As 
	$$\curl{W}(x)=- (\sin x_3,\,\sin x_1,\,\sin x_2) \quad\text{and}\quad \nabla{\psi}(x)=-\frac{x}{4\eta T}e^{-\frac{\abs{x}^2}{8\eta T}},$$
	we conclude that $\curl{\psi W}(0)=0$, so $0$ is a critical point for the field. Moreover, a direct computation shows that
	$$\nabla\curl{\psi W}(0) = \frac{1}{4\eta T}\left(\begin{array}{ccc}
		0 & -1 & 4\eta T+1 \\
		4\eta T+1 & 0 & -1 \\
		-1 & 4\eta T+1 & 0
	\end{array}\right),$$
	so
	$$\det(\nabla\curl{\psi W}(0))=1+\frac{3}{4\eta T}+\frac{3}{(4\eta T)^2}\neq0$$
	and it has the eigenvalues
	\begin{align*}
		\lambda_1 &= 1, \\
		\lambda_2 &= -\frac{1}{2} + \i\frac{\sqrt{3}}{2} \sqrt{1 + \frac{1}{\eta T} + \frac{1}{4(\eta T)^2}}, \\
		\lambda_3 &= \overline{\lambda_2}.
	\end{align*}
	Thus, $(0,0,0)$ is hyperbolic for the field.
\end{proof}

\begin{lem}[Size of the initial data]
	Let $(u_0,b_0)$ be as in \eqref{eq:def_large_initial_data}. Then, the size of the initial data is large in any critical space:
	$$\norm{u_0}{\dot{B}_{\infty,\infty}^{-1}} \simeq M \simeq \norm{b_0}{\dot{B}_{\infty,\infty}^{-1}}.$$
\end{lem}
\begin{proof}
	The proof for $\norm{u_0}{\dot{B}_{\infty,\infty}^{-1}}$ follows the same lines as in \cite[Lemma 3.7]{CL24}. We recall, for the reader convenience, the computation of the dominant part of the $\dot{B}_{\infty,\infty}^{-1}$ norm
	of $u_0$  
    \begin{equation*}
        \norm{M N e^{-\eta t N^2}(\phi B_N)}{\dot{B}_{\infty,\infty}^{-1}} 
        \simeq M N \sup_{t>0} \sqrt{t} e^{-\eta tN^2}\norm{\phi B_N}{L^\infty} \simeq M,
    \end{equation*}
    the implicit constant in the last equivalence being proportional to $\frac1{\sqrt{\eta}}$. 
The fact that this is the main part  of the $\dot{B}_{\infty,\infty}^{-1}$ norm
	of $u_0$ can be seen, recalling the caloric description  of Besov spaces given in Remark \ref{rmk:Besov.caloric}, 
	 from the expansion
   $$
e^{t\Delta}u_0 
 = N e^{-t \lambda^2} \phi    B_{N}  +  N [e^{t\Delta} , \phi]  B_{N}   
+ e^{t\Delta} ( \nabla\phi \wedge B_{N} ).
    $$
Then one can proceed as in
\cite[Lemma 3.7]{CL24} to show that the second and third terms behaves as reminders, 
after rescaling the function $\phi$ with an appropriate ($M$ dependent) factor $L$, ad done in \cite[Lemma 3.7]{CL24}.
This rescaling is consistent with the proof of reconnection that we give, that for simplicity we only present for $L=1$.

Regarding the $\dot{B}_{\infty,\infty}^{-1}$ norm of $b_0$ one must notice that 
$$
	\| u_0 - b_0 \|_{\dot{B}_{\infty,\infty}^{-1}} = \rho \| e^{-\eta T\Delta}\curl{\psi W} \|_{\dot{B}_{\infty,\infty}^{-1}}
$$
and take $\rho$ sufficiently small, depending on $M$ (as we will also do in the proof of the reconnection).

\end{proof}

\subsection{Proof of the reconnection}
Our goal is to prove that the  solution obtained by applying Theorem \ref{thm:global_existence_large_data} to the class of initial data defined in \eqref{eq:def_large_initial_data} shows magnetic reconnection. To do so, we will make use of the robust stability of the number of hyperbolic critical points that a field has and carry out a perturbative argument. The rest of the article is devoted to the proof the reconnection result:
\begin{thm2}\label{thm:large_magnetic_reconnection}
    Given any constants $\eta,M,\,T>0$, there exist smooth divergence-free vector fields $u_0,\,b_0$ in $\R^3$ such that $\norm{u_0}{\Dot{B}_{\infty,\infty}^{-1}} \simeq \norm{b_0}{\Dot{B}_{\infty,\infty}^{-1}} \simeq M$ and \eqref{eq:MHD} admits a (unique) global strong solution $(u,b)$ with initial datum $(u_0,b_0)$, such that the magnetic lines at time $t=0$ and $t=T$ are not topologically equivalent, meaning that there is no homeomorphism of $\R^3$ into itself mapping the magnetic lines of $b(0,\cdot)$ into those of $b(T,\cdot)$.
\end{thm2}
\begin{proof}
Notice that the perturbation $h$ defined in \eqref{eq:def_perturbations} satisfies
$$h(T) = b(T)-e^{\eta T\Delta}b_0 = b(T)-Me^{\eta T\Delta}\curl{\phi B_N}-\rho\curl{\psi W}.$$
On the other hand, by Duhamel's formula, $h(T)=e^{\eta T\Delta}h_0+D_h(T) = D_h(T)$ (since $h_0=0$), so we can put together both equations and arrive to:
\begin{equation}\label{eq:main_eq}
	\frac{1}{\rho}b(T)-\curl{\psi W} = \frac{M}{\rho}e^{\eta T\Delta}\curl{\phi B_N} + \frac{D_h(T)}{\rho}\,,
\end{equation}
where
\begin{subequations}\label{eq:Duhamel_h}
	\begin{align}
		D_h(t) &= \int_0^t e^{\eta(t-s)\Delta}\div\left(h(s)\otimes v(s) - v(s)\otimes h(s)\right)\,\d s \tag{$D_h^1$}\label{eq:Duhamel_h1}\\
		&+ \int_0^t e^{\eta(t-s)\Delta}\div\left(-e^{\eta s\Delta}u_0\otimes e^{\eta s\Delta}b_0 + e^{\eta s\Delta}b_0\otimes e^{\eta s\Delta}u_0\right)\,\d s \tag{$D_h^2$}\label{eq:Duhamel_h2}\\
		&+ \int_0^t e^{\eta (t-s)\Delta}\div\left(-e^{\eta s\Delta}u_0\otimes h(s) + e^{\eta s\Delta}b_0\otimes v(s) - v(s)\otimes e^{\eta t\Delta}b_0 + h(s)\otimes e^{\eta s\Delta}u_0\right)\,\d s\,. \tag{$D_h^3$}\label{eq:Duhamel_h3}
	\end{align}
\end{subequations}
If we could show that the $C^1$-norm of the left-hand side of \eqref{eq:main_eq} is small, then we can guarantee by the implicit function theorem that $b(T)$ also has a hyperbolic critical point. On the other hand, $b_0$ has no critical points. Indeed, one can bound from below
$$\abs{b_0}>M\abs{\curl{\phi B_N}}-\rho\abs{e^{-\eta T\Delta}\curl{\psi W}}.$$
Now, observe that $e^{-\eta T\Delta}\curl{\psi W}\in L^\infty$ by \eqref{eq:estim_norm_evolution_psiW_Lp_small_times} with $t=0$ and Sobolev embedding (we are considering in fact $r \geq 3$), and that $\abs{\curl{\phi B_N}}>0$ by Lemma \ref{lem:curl.phi.BN.crit.points} for suitable choices of $\alpha$ and $N$. Then, one can choose $\rho>0$ sufficiently small (depending on $M,\, N,\, T,\, \eta$) so that $\abs{b_0(x)}>0$ for all $x\in\R^3$. Therefore, our goal is to make
\begin{equation*}
    \norm{\frac{1}{\rho}b(T)-\curl{\psi W}}{C^1} \ll 1,
\end{equation*}
which is the same as bounding the right-hand side of \eqref{eq:main_eq} by a very small constant. Thanks to Sobolev embeddings it will be sufficient to obtain a bound for the $H^r$-norm of this quantity for some $r>5/2$.

\medskip

First of all, we observe that the field $e^{\eta T\Delta}\curl{\phi B_N}$ satisfies the bound
\begin{equation*}
	\begin{split}
		\norm{e^{\eta T\Delta}\curl{\phi B_N}}{H^r}^2 &\simeq \norm{e^{\eta T\Delta}\curl{\phi B_N}}{L^2}^2 + \norm{e^{\eta T\Delta}\nabla^r\curl{\phi B_N}}{L^2}^2 \\
		&\lesssim e^{-2c\eta TN^2}N^{2(r+1)}+N^{-2n}
	\end{split}
\end{equation*}
thanks to the estimates \eqref{eq:estim_norm_evolution_phiBN_Lp_small_times} and \eqref{eq:estim_norm_evolution_alpha-phiBN_Lp_small_times}. We will see later that this $H^r$-bound is sufficiently small for our purposes (choosing the suitable parameters). Thus, we are going to focus on bounding the Duhamel term in \eqref{eq:main_eq}. For the integral \eqref{eq:Duhamel_h2}, we write the initial data in explicit form and observe that everything is proportional to $\rho$:
\begin{equation*}
    \begin{split}
        \left(e^{\eta t\Delta}b_0\cdot\nabla\right)e^{\eta t\Delta}u_0 - \left(e^{\eta t\Delta}u_0\cdot\nabla\right) e^{\eta t\Delta}b_0 & = M\rho \left(e^{\eta (t-T)\Delta}\curl{\psi W}\cdot\nabla\right)e^{\eta t\Delta}\curl{\phi B_N} \\
        &\quad -M\rho\left(e^{\eta t\Delta}\curl{\phi B_N}\cdot\nabla\right)e^{\eta (t-T)\Delta}\curl{\psi W}.
    \end{split}
\end{equation*}
Thus, we just need to bound the $H^r$-norm of \eqref{eq:Duhamel_h2} by something small by itself, independently of $\rho$. To do so, we need the following estimate:
\begin{lem}\label{lem:Hr_to_L2}
	Let $r>0$ and $t>0$. For any $f\in L^2$, one has
	$$\norm{e^{\eta t\Delta}f}{H^r} \lesssim (1+(\eta t)^{-r})^{1/2}\norm{f}{L^2}.$$
\end{lem}
\begin{proof}
	First of all,
	$$\norm{e^{\eta t\Delta}f}{\Dot{H}^r}^2 \simeq \int_{\R^3} \abs{\xi}^{2r}e^{-2\eta t\abs{\xi}^2}\abs{\hat{f}(\xi)}^2\,\d\xi \lesssim (\eta t)^{-r}\int_{\R^3}\abs{\hat{f}(\xi)}^2\,\d\xi \lesssim (\eta t)^{-r}\norm{f}{L^2}^2.$$
	Here we have used that $e^x=\sum_{k\geq0} \frac{x^k}{k!}\geq \frac{x^r}{r!}$ for any integer $r\geq0$ and $x>0$. By the continuity of the heat kernel on $L^2$ and using that
	$$\norm{e^{\eta t\Delta}f}{H^r}^2 = \norm{e^{\eta t\Delta}f}{L^2}^2 + \norm{e^{\eta t\Delta}f}{\Dot{H}^r}^2,$$
	the result follows.
\end{proof}
Using this lemma and the heat kernel estimates,
\begin{subequations}\label{eq:estim_Hr_D_h2}
	\begin{align}
		\norm{D_h^2(T)}{H^r} &\lesssim M\rho\int_0^T \norm{e^{\eta (T-s)\Delta} \div\left[e^{\eta s\Delta}\curl{\phi B_N}\otimes e^{\eta (s-T)\Delta}\curl{\psi W}\right]}{H^r}\d s \notag\\
		&\quad+ M\rho\int_0^T\norm{e^{\eta (T-s)\Delta}\div \left[e^{\eta (s-T)\Delta}\curl{\psi W}\otimes e^{\eta s\Delta}\curl{\phi B_N}\right]}{H^r}\d s \notag\\
		&\lesssim \rho\int_0^{T/2} (1+(\eta (T-s))^{-r-1})^{1/2}\norm{e^{\eta s\Delta}\curl{\phi B_N}}{L^2}\norm{e^{\eta (s-T)\Delta}\curl{\psi W}}{L^\infty}\,\d s \label{eq:int_D_h2_0_T2}\\
		&\quad+ \rho\int_{T/2}^T (\eta (T-s))^{-1/2}\norm{e^{\eta s\Delta}\curl{\phi B_N}}{H^r}\norm{e^{\eta (s-T)\Delta}\curl{\psi W}}{H^r}\,\d s. \label{eq:int_D_h2_T2_T}
	\end{align}
\end{subequations}
For the first integral, we bound the time factor by a factor depending on $T$ and we deal with the norms using Lemma \ref{lem:norm_estimates_low_times}. Going back to \eqref{eq:int_D_h2_0_T2}, this gives us that
\begin{align*}
	\rho\int_0^{T/2} (1+(\eta (T-s))^{-r-1})^{1/2}&\norm{e^{\eta s\Delta}\curl{\phi B_N}}{L^2}\norm{e^{\eta (s-T)\Delta}\curl{\psi W}}{L^\infty}\,\d s \\
	&\lesssim \rho\left(1+(\eta T/2)^{-r-1}\right)^{1/2} N \int_0^{T/2} e^{-c\eta sN^2}\,\d s \\
	&\quad+ \rho\left(1+(\eta T/2)^{-r-1}\right)^{1/2}  N^{-n}\int_0^{T/2}\,\d s \\
	&\lesssim  \rho\,\eta^{-\frac{r+1}{2}}\frac{1-e^{-c\eta TN^2/2}}{N} + \rho\,\eta^{-\frac{r+1}{2}} N^{-n}.
\end{align*}
On the other hand, \eqref{eq:int_D_h2_T2_T} is bounded in the following way:
\begin{equation*}
	\begin{split}
		\rho\int_{T/2}^T (\eta (T-s))^{-1/2}\norm{e^{s\Delta}\curl{\phi B_N}}{H^r}&\norm{e^{s\Delta}\curl{\psi W}}{H^r}\,\d s \\
		&\lesssim \rho\,\eta^{-1/2} N^{r+1}\int_{T/2}^T (T-s)^{-1/2}e^{-c\eta sN^2}\,\d s \\
		&\quad+ \rho\,\eta^{-1/2} N^{-n}\int_{T/2}^T (T-s)^{-1/2}\,\d s.
	\end{split}
\end{equation*}
We compute
$$\int_{T/2}^T (T-s)^{-1/2}e^{-c\eta sN^2}\,\d s \lesssim e^{-c\eta TN^2/2} \int_{T/2}^T (T-s)^{-1/2}\,\d s \lesssim e^{-c\eta TN^2/2}T^{1/2}$$
and so
$$\norm{D_h^2(T)}{H^r} \lesssim \rho\frac{1-e^{-c\eta TN^2/2}}{N} + \rho N^{-n} + \rho e^{-c\eta TN^2/2}N^{r+1} \lesssim \rho (N^{-1} + e^{-c\eta TN^2/2}N^{r+1}),$$
where the constant depends on $n$, $M$ and $T$, and on a negative power of $\eta$. We can choose $N$ and $n$ sufficiently large so that this Duhamel term is small enough for the perturbative analysis. For a choice of the parameters see the last part of the proof.

\medskip

It remains to bound \eqref{eq:Duhamel_h1} and \eqref{eq:Duhamel_h3}, for which we will need the \textit{a priori} estimates on the $H^r$-norms of $v$ and $h$ that we obtained in Theorem \ref{thm:global_existence_large_data}. From now on, the explicit dependence of the constants on $\eta$ will be omitted, as it is the same as in the previous computations. For the first Duhamel term,
\begin{align*}
    \norm{D_h^1(T)}{H^r} &\lesssim \int_0^{T/2} (1+(T-s)^{-r-1})^{1/2}\norm{h(s)\otimes v(s)-v(s)\otimes h(s)}{L^2}\,\d s \\
    &\quad+ \int_{T/2}^T (T-s)^{-1/2}\norm{h(s)\otimes v(s)-v(s)\otimes h(s)}{H^r}\,\d s \\
    &\lesssim \int_0^{T/2} \norm{h(s)}{H^2}\norm{v(s)}{H^2}\,\d s + \int_{T/2}^T (T-s)^{-1/2} \norm{h(s)}{H^r}\norm{v(s)}{H^r}\,\d s \\
    &\lesssim \rho^2N^{4}\int_0^{T/2}\,\d s + \rho^2N^{2r}\int_{T/2}^T (T-s)^{-1/2}\,\d s \\
   &\lesssim \rho^2N^{2(r+2)},
\end{align*}
where in the first inequality we have made use of Lemma \ref{lem:Hr_to_L2} and the heat kernel estimates in $H^r$; in the second inequality we have used the embedding $H^2\hookrightarrow L^\infty$ and the algebra property of $H^r(\R^3)$ for $r>3/2$; and in the last inequality we have used the energy estimates of Theorem \ref{thm:global_existence_large_data}. Finally, for the third Duhamel term,
\begin{align*}
    \norm{D_h^3(T)}{H^r} &\lesssim \int_0^{T/2} (1+(T-s)^{-r-1})^{1/2}\norm{e^{\eta s\Delta}u_0\otimes h(s) + e^{\eta s\Delta}b_0\otimes v(s)}{L^2}\,\d s \\
    &\quad+ \int_{T/2}^T (T-s)^{-1/2} \norm{e^{\eta s\Delta}u_0\otimes h(s) + e^{\eta s\Delta}b_0\otimes v(s)}{H^r}\,\d s \\
    &\lesssim \int_0^{T/2} \left(\norm{e^{\eta s\Delta}u_0}{L^\infty}\norm{h(s)}{L^2} + \norm{e^{\eta s\Delta}b_0}{L^\infty}\norm{v(s)}{L^2} \right)\,\d s \\
    &\quad+ \int_{T/2}^T (T-s)^{-1/2} \left(\norm{e^{\eta s\Delta}u_0}{H^r}\norm{h(s)}{H^r} + \norm{e^{\eta s\Delta}b_0}{H^r}\norm{v(s)}{H^r}\right)\,\d s \\
    &\lesssim \rho \int_0^{T/2} \left(e^{-c\eta sN^2}N + N^{-n}\right)\,\d s \\
    &\quad+ \rho N^{r}\int_{T/2}^T (T-s)^{-1/2} \left(e^{-c\eta sN^2}N^{r+1} + N^{-n}\right)\,\d s \\
    &\lesssim \rho \left(N^{-1}+N^{-n}\right) + \rho N^r(e^{-c\eta TN^2/2}N^{r+1} + N^{-n}).
\end{align*}
Now that we have bounded all Duhamel terms, we return to \eqref{eq:main_eq}:
\begin{equation*}
	\begin{split}
		\norm{\frac{1}{\rho}b(T)-\curl{\psi W}}{H^r} &\leq \norm{\frac{M}{\rho}e^{\eta T\Delta}\curl{\phi B_N}}{H^r} + \sum_{i=1}^3\norm{\frac{D_h^i(T)}{\rho}}{H^r} \\
		&\lesssim \frac{1}{\rho}\left(e^{-c\eta TN^2}N^{r+1}+ N^{-n}\right) \\
		&\quad + \frac{1}{\rho}\left(\rho^2N^{2(r+2)} + \rho(N^{-1}+N^{-n}) + \rho N^{2r}(e^{-c\eta TN^2/2}+N^{-n}) \right).
	\end{split}
\end{equation*}
For the choice of the parameters, taking $\rho=N^{-\beta}$ for some $\beta>0$ we get that
$$\norm{\frac{1}{\rho}b(T)-\curl{\psi W}}{C^1} \leq e^{-c\eta TN^2}N^{r+1+\beta} + N^{\beta-n} + N^{2(r+2)-\beta} + N^{-1}+e^{-c\eta TN^2/2}N^{2r}+N^{2r-n},$$
and everything can be made small if we choose 
\begin{equation}\label{fmdnjskjhgfjskjgb2r+6}
\beta=2(r+2)+1, \qquad  n=2r+6.
\end{equation} and we make $N$ sufficiently large so that the exponential decay dominates the polynomial growth in $N$.

\medskip

To conclude the proof, remember that we want to ensure that $b(T,\cdot)$ has a hyperbolic critical point. For that, choose $r=3>5/2$, so that we can use Sobolev embbedings to bound
$$\norm{\frac{b(T)}{\rho}-\curl{\psi W}}{C^1} \lesssim \norm{\frac{M}{\rho}e^{T\Delta}\curl{\phi B_N}-\frac{D_h(T)}{\rho}}{H^r} \ll 1 \quad\text{for } N\gg1$$
and this is enough to guarantee that $b(T)$ has a hyperbolic critical point near $(0,0,0)$, as 
hyperbolic critical points are stable under $C^1$-perturbations.
\end{proof}

\begin{remark}
	Two remarkable facts about the dependence on the resistivity may be noticed:
	\begin{itemize}
		 \item The constant that appears at the end of the proof is inversely proportional to $\eta$, so it blows up as $\eta\rightarrow0$. This is consistent with the ideal scenario, where the magnetic lines are \textit{frozen-in-flux}.
		 \item In order for the quantity $\delta_0(N)$ to be small, the timescale of reconnection is $T=O(\eta^{-1}N^{-2})$. The reconnection is therefore not instantaneous.
	\end{itemize}
\end{remark}

\section{Acknowledgements}
C. Peña-Vázquez de la Torre is supported by BERC program, by the project PID2021-123034NB-I00 funded by MICIU/AEI/10.13039/501100011033  and FEDER/EU and by the Severo Ochoa accreditation CEX2021-00142-S (BCAM). \\
Renato Lucà is supported by BERC program and by MICINN (Spain) projects Severo Ochoa CEX2021-001142, PID2021-123034NB-I00
funded by MICIU/AEI/10.13039/501100011033  and FEDER/EU.

\bibliographystyle{alpha}

\bibliography{references3}

@article{Al42,
	author = {Alfv{\'e}n, H. },
	date = {1942/10/01},
	doi = {10.1038/150405d0},
	id = {ALFV{\'E}N1942},
	isbn = {1476-4687},
	journal = {Nature},
	number = {3805},
	pages = {405--406},
	title = {Existence of Electromagnetic-Hydrodynamic Waves},
	url = {https://doi.org/10.1038/150405d0},
	volume = {150},
	year = {1942}}

@book{BCD11,
  title={Fourier Analysis and Nonlinear Partial Differential Equations},
  author={Bahouri, H. and Chemin, J.Y. and Danchin, R.},
  isbn={9783642168307},
  series={Grundlehren der mathematischen Wissenschaften},
  url={https://books.google.es/books?id=CcTnaveQkn0C},
  year={2011},
  publisher={Springer Berlin Heidelberg}
}

@book{Bre10,
  title={Functional Analysis, Sobolev Spaces and Partial Differential Equations},
  author={Brezis, H.},
  isbn={9780387709130},
  lccn={2010938382},
  series={Universitext},
  url={https://books.google.es/books?id=GAA2XqOIIGoC},
  year={2010},
  publisher={Springer New York}
}

@article{BSS88,
 ISSN = {00029947, 10886850},
 URL = {http://www.jstor.org/stable/2001047},
 author = {C. Bardos and C. Sulem and P. L. Sulem},
 journal = {Transactions of the American Mathematical Society},
 number = {1},
 pages = {175--191},
 publisher = {American Mathematical Society},
 title = {Longtime Dynamics of a Conductive Fluid in the Presence of a Strong Magnetic Field},
 urldate = {2026-04-14},
 volume = {305},
 year = {1988}
}

@article{CCL25,
    author  = {Caro, Pedro and Ciampa, Gennaro and Luc\`{a}, Renato},
    title   = {Large time behavior of magnetic reconnection in the magnetohydrodynamics equations},
    journal = {Revista Matem\'{a}tica Iberoamericana},
    year    = {2025},
    volume  = {41},
    number  = {2},
    pages   = {461--508},
    doi     = {10.4171/RMI/1512},
    url     = {https://ems.press/journals/rmi/articles/14298348}
}

@article{CDL14,
     author = {Chae, Dongho and Degond, Pierre and Liu, Jian-Guo},
     title = {Well-posedness for {Hall-magnetohydrodynamics}},
     journal = {Annales de l'Institut Henri Poincar\'e. C, Analyse non lin\'eaire},
     pages = {555--565},
     year = {2014},
     publisher = {Elsevier},
     volume = {31},
     number = {3},
     doi = {10.1016/j.anihpc.2013.04.006},
     mrnumber = {3208454},
     zbl = {1297.35064},
     language = {en},
     url = {https://numdam.org/articles/10.1016/j.anihpc.2013.04.006/}
}

@article{CF21,
title = {Rigorous derivation and well-posedness of a quasi-homogeneous ideal MHD system},
journal = {Nonlinear Analysis: Real World Applications},
volume = {60},
pages = {103284},
year = {2021},
issn = {1468-1218},
doi = {https://doi.org/10.1016/j.nonrwa.2020.103284},
url = {https://www.sciencedirect.com/science/article/pii/S1468121820302029},
author = {Dimitri Cobb and Francesco Fanelli},
}

@article{CF23,
title = {Elsässer formulation of the ideal MHD and improved lifespan in two space dimensions},
journal = {Journal de Mathématiques Pures et Appliquées},
volume = {169},
pages = {189-236},
year = {2023},
issn = {0021-7824},
doi = {https://doi.org/10.1016/j.matpur.2022.11.012},
url = {https://www.sciencedirect.com/science/article/pii/S0021782422001556},
author = {Dimitri Cobb and Francesco Fanelli}
}

@article{CG09,
title = {Wellposedness and stability results for the Navier–Stokes equations in R3},
journal = {Annales de l'Institut Henri Poincaré C, Analyse non linéaire},
volume = {26},
number = {2},
pages = {599-624},
year = {2009},
issn = {0294-1449},
doi = {https://doi.org/10.1016/j.anihpc.2007.05.008},
url = {https://www.sciencedirect.com/science/article/pii/S0294144908000140},
author = {Jean-Yves Chemin and Isabelle Gallagher}
}

@incollection{Ci24,
    author    = {Ciampa, Gennaro},
    title     = {On the Topology of the Magnetic Lines of Solutions of the {MHD} Equations},
    booktitle = {Hyperbolic Problems: Theory, Numerics, Applications. Volume I. HYP 2022},
    series    = {SEMA SIMAI Springer Series},
    volume    = {34},
    publisher = {Springer},
    address   = {Cham},
    year      = {2024},
    pages     = {193--203},
    doi       = {10.1007/978-3-031-54400-2_14}
}

@article{CL24,
title = {Localization of Beltrami fields: Global smooth solutions and vortex reconnection for the Navier-Stokes equations},
journal = {Journal of Functional Analysis},
volume = {287},
number = {9},
pages = {110610},
year = {2024},
issn = {0022-1236},
doi = {https://doi.org/10.1016/j.jfa.2024.110610},
url = {https://www.sciencedirect.com/science/article/pii/S0022123624002982},
author = {Gennaro Ciampa and Renato Lucà}
}

@book{Da01, 
place={Cambridge}, 
series={Cambridge Texts in Applied Mathematics}, 
title={An Introduction to Magnetohydrodynamics}, 
publisher={Cambridge University Press}, 
author={Davidson, P. A.}, year={2001}, 
collection={Cambridge Texts in Applied Mathematics}
}

@article{DL72,
	author = {Duvaut, G. and Lions, J. L.},
	date = {1972/01/01},
	doi = {10.1007/BF00250512},
	id = {Duvaut1972},
	isbn = {1432-0673},
	journal = {Archive for Rational Mechanics and Analysis},
	number = {4},
	pages = {241--279},
	title = {In{\'e}quations en thermo{\'e}lasticit{\'e}et magn{\'e}tohydrodynamique},
	url = {https://doi.org/10.1007/BF00250512},
	volume = {46},
	year = {1972}
}

@article{EP12,
    author  = {Enciso, Alberto and Peralta-Salas, Daniel},
    title   = {Knots and links in steady solutions of the {E}uler equation},
    journal = {Annals of Mathematics},
    year    = {2012},
    volume  = {175},
    number  = {1},
    pages   = {345--367},
    publisher = {Princeton University},
    doi     = {10.4007/annals.2012.175.1.9},
    url     = {https://annals.math.princeton.edu/2012/175-1/p09}
}

@article{EP15,
	author = {Enciso, Alberto and Peralta-Salas, Daniel},
	date = {2015/03/01},
	doi = {10.1007/s11511-015-0123-z},
	id = {Enciso2015},
	isbn = {1871-2509},
	journal = {Acta Mathematica},
	number = {1},
	pages = {61--134},
	title = {Existence of knotted vortex tubes in steady Euler flows},
	url = {https://doi.org/10.1007/s11511-015-0123-z},
	volume = {214},
	year = {2015}
    }

@article{EP25,
	author = {Enciso, Alberto and Peralta-Salas, Daniel},
	date = {2024/12/10},
	doi = {10.1007/s00205-024-02078-5},
	id = {Enciso2024},
	isbn = {1432-0673},
	journal = {Archive for Rational Mechanics and Analysis},
	number = {1},
	pages = {6},
	title = {Obstructions to Topological Relaxation for Generic Magnetic Fields},
	url = {https://doi.org/10.1007/s00205-024-02078-5},
	volume = {249},
	year = {2024}}

@article{ELP17,
title = {Vortex reconnection in the three dimensional Navier–Stokes equations},
journal = {Advances in Mathematics},
volume = {309},
pages = {452-486},
year = {2017},
issn = {0001-8708},
doi = {https://doi.org/10.1016/j.aim.2017.01.025},
url = {https://www.sciencedirect.com/science/article/pii/S0001870816312403},
author = {Alberto Enciso and Renato Lucà and Daniel Peralta-Salas}
}

@article{EPR23, 
title={Beltrami fields exhibit knots and chaos almost surely}, 
volume={11}, 
DOI={10.1017/fms.2023.52}, journal={Forum of Mathematics, Sigma}, 
author={Enciso, Alberto and Peralta-Salas, Daniel and Romaniega, Álvaro}, 
year={2023}, 
pages={e56}}

@article{FMRR14,
title = {Higher order commutator estimates and local existence for the non-resistive MHD equations and related models},
journal = {Journal of Functional Analysis},
volume = {267},
number = {4},
pages = {1035-1056},
year = {2014},
issn = {0022-1236},
doi = {https://doi.org/10.1016/j.jfa.2014.03.021},
url = {https://www.sciencedirect.com/science/article/pii/S0022123614001505},
author = {Charles L. Fefferman and David S. McCormick and James C. Robinson and Jose L. Rodrigo}
}

@article{GPWXH22,
author = {Guo, Ruilong and Pu, Zuyin and Wang, Xiaogang and Xiao, Chijie and He, Jiansen},
title = {3D Reconnection Geometries With Magnetic Nulls: Multispacecraft Observations and Reconstructions},
journal = {Journal of Geophysical Research: Space Physics},
volume = {127},
number = {2},
pages = {e2021JA030248},
doi = {https://doi.org/10.1029/2021JA030248},
url = {https://agupubs.onlinelibrary.wiley.com/doi/abs/10.1029/2021JA030248},
eprint = {https://agupubs.onlinelibrary.wiley.com/doi/pdf/10.1029/2021JA030248},
note = {e2021JA030248 2021JA030248},
year = {2022}
}

@article{HHW14,
  title={On some new global existence results for 3D magnetohydrodynamic equations},
  author={Cheng He and Xiangdi Huang and Yun Wang},
  journal={Nonlinearity},
  year={2013},
  volume={27},
  pages={343 - 352},
  url={https://api.semanticscholar.org/CorpusID:119167239}
}

@article{HXY18,
author = {He, Lingbing and Xu, Li and Yu, Pin},
year = {2017},
month = {12},
pages = {},
title = {On Global Dynamics of Three Dimensional Magnetohydrodynamics: Nonlinear Stability of Alfvén Waves},
volume = {4},
journal = {Annals of PDE},
doi = {10.1007/s40818-017-0041-9}
}

@article{JV18,
    author        = {Jafari, Amir and Vishniac, Ethan},
    title         = {Introduction to Magnetic Reconnection},
    journal       = {arXiv preprint arXiv:1805.01347},
    year          = {2018},
    eprint        = {1805.01347},
    archivePrefix = {arXiv},
    primaryClass  = {astro-ph.SR},
    url           = {https://arxiv.org/abs/1805.01347}
}

@article{JZ16,
	author = {Jia, Xuanji and Zhou, Yong},
	date = {2016/03/01},
	doi = {10.1007/s00021-015-0246-1},
	id = {Jia2016},
	isbn = {1422-6952},
	journal = {Journal of Mathematical Fluid Mechanics},
	number = {1},
	pages = {187--206},
	title = {On Regularity Criteria for the 3D Incompressible MHD Equations Involving One Velocity Component},
	url = {https://doi.org/10.1007/s00021-015-0246-1},
	volume = {18},
	year = {2016}}

@article{La20,
    author = {Lazarian, Alex and Eyink, Gregory L. and Jafari, Amir and Kowal, Grzegorz and Li, Hui and Xu, Siyao and Vishniac, Ethan T.},
    title = {3D turbulent reconnection: Theory, tests, and astrophysical implications},
    journal = {Physics of Plasmas},
    volume = {27},
    number = {1},
    pages = {012305},
    year = {2020},
    month = {01},
    issn = {1070-664X},
    doi = {10.1063/1.5110603},
    url = {https://doi.org/10.1063/1.5110603},
    eprint = {https://pubs.aip.org/aip/pop/article-pdf/doi/10.1063/1.5110603/16024110/012305_1_online.pdf},
}

@inbook{LS12a,
title = "Bifurcations of solutions of the 2-dimensional Navier-Stokes system",
author = "Dong Li and Sinai, \{Yakov G.\}",
note = "Publisher Copyright: {\textcopyright} Springer-Verlag Berlin Heidelberg 2012.",
year = "2012",
month = jan,
day = "1",
doi = "10.1007/978-3-642-28821-0\_10",
language = "English (US)",
isbn = "9783642288203",
pages = "241--269",
booktitle = "Essays in Mathematics and its Applications",
publisher = "Springer Berlin Heidelberg",
}

@article{LS12b,
  title={Nonsymmetric bifurcations of solutions of the 2D Navier–Stokes system},
  author={Dong Li and Yakov G. Sinai},
  journal={Advances in Mathematics},
  year={2012},
  volume={229},
  pages={1976-1999},
  url={https://api.semanticscholar.org/CorpusID:120891595}
}

@article{Ni59,
     author = {Nirenberg, L.},
     title = {On elliptic partial differential equations},
     journal = {Annali della Scuola Normale Superiore di Pisa - Scienze Fisiche e Matematiche},
     pages = {115--162},
     year = {1959},
     publisher = {Scuola normale superiore},
     volume = {Ser. 3, 13},
     number = {2},
     mrnumber = {109940},
     zbl = {0088.07601},
     language = {en},
     url = {https://www.numdam.org/item/ASNSP_1959_3_13_2_115_0/}
}

@article{PP22,
	author = {Pontin, David I. and Priest, Eric R.},
	date = {2022/05/10},
	doi = {10.1007/s41116-022-00032-9},
	id = {Pontin2022},
	isbn = {1614-4961},
	journal = {Living Reviews in Solar Physics},
	number = {1},
	pages = {1},
	title = {Magnetic reconnection: MHD theory and modelling},
	url = {https://doi.org/10.1007/s41116-022-00032-9},
	volume = {19},
	year = {2022}}

@book{PT00, 
place={Cambridge}, 
title={Magnetic Reconnection: MHD Theory and Applications}, 
publisher={Cambridge University Press}, 
author={Priest, Eric and Forbes, Terry}, 
year={2000}
}

@book{RRS16, 
place={Cambridge}, 
series={Cambridge Studies in Advanced Mathematics}, 
title={The Three-Dimensional Navier–Stokes Equations: Classical Theory}, 
publisher={Cambridge University Press}, 
author={Robinson, James C. and Rodrigo, José L. and Sadowski, Witold}, 
year={2016}, 
collection={Cambridge Studies in Advanced Mathematics}}

@article{SM88,
title = {On a magnetohydrodynamic problem of Euler type},
journal = {Journal of Differential Equations},
volume = {74},
number = {2},
pages = {318-335},
year = {1988},
issn = {0022-0396},
doi = {https://doi.org/10.1016/0022-0396(88)90008-3},
url = {https://www.sciencedirect.com/science/article/pii/0022039688900083},
author = {Paul Günter Schmidt}
}

@article{ST83,
    author  = {Sermange, Michel and Temam, Roger},
    title   = {Some mathematical results concerning the equations of magnetohydrodynamics},
    journal = {Communications on Pure and Applied Mathematics},
    year    = {1983},
    volume  = {36},
    number  = {5},
    pages   = {635--664},
    doi     = {10.1002/cpa.3160360506},
    url     = {https://onlinelibrary.wiley.com/doi/abs/10.1002/cpa.3160360506}
}

@article{T09,
doi = {10.1088/0004-637X/693/1/1029},
url = {https://doi.org/10.1088/0004-637X/693/1/1029},
year = {2009},
month = {mar},
publisher = {The American Astronomical Society},
volume = {693},
number = {1},
pages = {1029},
author = {Titov, V. S. and Forbes, T. G. and Priest, E. R. and Mikić, Z. and Linker, J. A.},
title = {SLIP-SQUASHING FACTORS AS A MEASURE OF THREE-DIMENSIONAL MAGNETIC RECONNECTION},
journal = {The Astrophysical Journal}
}

@article{DQS11,
	author = {Mimi Dai and Jie Qing and Maria E. Schonbek},
	title = {{Norm inflation for incompressible magneto-hydrodynamic system in $\dot{B}_{\infty}^{-1,\infty}$}},
	volume = {16},
	journal = {Advances in Differential Equations},
	number = {7/8},
	publisher = {Khayyam Publishing, Inc.},
	pages = {725 -- 746},
	year = {2011},
	doi = {10.57262/ade/1355703204},
	URL = {https://doi.org/10.57262/ade/1355703204}
}

@article{CD20,
	title = {Discontinuity of weak solutions to the 3D NSE and MHD equations in critical and supercritical spaces},
	journal = {Journal of Mathematical Analysis and Applications},
	volume = {481},
	number = {2},
	pages = {123493},
	year = {2020},
	issn = {0022-247X},
	doi = {https://doi.org/10.1016/j.jmaa.2019.123493},
	url = {https://www.sciencedirect.com/science/article/pii/S0022247X19307619},
	author = {Alexey Cheskidov and Mimi Dai}
}

@article{MYZ07,
	author = {Miao, Changxing and Yuan, Baoquan and Zhang, Bo},
	title = {Well-posedness for the incompressible magneto-hydrodynamic system},
	journal = {Mathematical Methods in the Applied Sciences},
	volume = {30},
	number = {8},
	pages = {961-976},
	doi = {https://doi.org/10.1002/mma.820},
	url = {https://onlinelibrary.wiley.com/doi/abs/10.1002/mma.820},
	eprint = {https://onlinelibrary.wiley.com/doi/pdf/10.1002/mma.820},
	year = {2007}
}

\end{document}